\documentclass[11pt]{amsart}

\usepackage{fullpage,cite,hyperref,enumitem}
\usepackage{amssymb,amsmath,amsthm,mathtools,extpfeil}
\usepackage[all,cmtip]{xy}
\usepackage{rotating,graphicx,subcaption}
\usepackage{tikz}

\newtheorem{thm}{Theorem}[section]
\newtheorem*{thm*}{Theorem}
\newtheorem{thmA}{Theorem}

\newtheorem{lem}[thm]{Lemma}
\newtheorem*{lem*}{Lemma}
\newtheorem{prop}[thm]{Proposition}
\newtheorem*{prop*}{Proposition}
\newtheorem{cor}[thm]{Corollary}

\theoremstyle{definition}
\newtheorem*{defn}{Definition}
\newtheorem*{rmk}{Remark}

\DeclareMathOperator{\id}{id}
\DeclareMathOperator{\Hom}{Hom}
\DeclareMathOperator{\Lip}{Lip}
\DeclareMathOperator{\Map}{Map}
\DeclareMathOperator{\Jac}{Jac}
\DeclareMathOperator{\vol}{vol}
\DeclareMathOperator{\img}{im}
\DeclareMathOperator{\st}{st}
\DeclareMathOperator{\mass}{mass}
\DeclareMathOperator{\geo}{geo}

\DeclareMathOperator{\len}{length}

\DeclareMathOperator{\mult}{mult}

\newcommand{\ph}{\varphi}
\newcommand{\epsi}{\varepsilon}

\begin{document}
\title{Quantitative null-cobordism}
\author{Gregory R. Chambers}
\address{Department of Mathematics, University of Chicago, Chicago, Illinois, USA}
\email[G.~R.~Chambers]{chambers@math.uchicago.edu}
\author{Dominic Dotterrer}
\address{Department of Computer Science, Stanford University, Stanford, California, USA}
\email[D.~Dotterrer]{dominicd@cs.stanford.edu}
\author{Fedor Manin}
\address{Department of Mathematics, University of Toronto, Toronto, Ontario, Canada}
\email[F.~Manin]{manin@math.toronto.edu}
\author{Shmuel Weinberger}
\address{Department of Mathematics, University of Chicago, Chicago, Illinois, USA}
\email[S.~Weinberger]{shmuel@math.uchicago.edu}

\begin{abstract}
  For a given null-cobordant Riemannian $n$-manifold, how does the minimal
  geometric complexity of a null-cobordism depend on the geometric complexity of
  the manifold?  In \cite{GrQHT}, Gromov conjectured that this dependence should
  be linear.  We show that it is at most a polynomial whose degree depends on
  $n$.

  This construction relies on another of independent interest.  Take $X$ and $Y$
  to be sufficiently nice compact metric spaces, such as Riemannian manifolds or
  simplicial complexes.  Suppose $Y$ is simply connected and rationally homotopy
  equivalent to a product of Eilenberg--MacLane spaces: for example, any simply
  connected Lie group.  Then two homotopic $L$-Lipschitz maps $f,g:X \to Y$ are
  homotopic via a $CL$-Lipschitz homotopy.  We present a counterexample to show
  that this is not true for larger classes of spaces $Y$.

  
  

\end{abstract}
\maketitle
\tableofcontents
\section{Introduction}

This paper is about two intimately related problems. One of them is quantitative
algebraic topology: using powerful algebraic methods, we frequently know a lot
about the homotopy classes of maps from one space to another, but these methods
are extremely indirect, and it’s hard to understand much about what these maps
look like or how the homotopies come to be. The other is the analogous problem
in geometric topology. The paradigm of this subject since immersion theory,
cobordism, surgery etc.~has been to take geometric problems and relate them to
problems in homotopy theory, and sometimes, algebraic K-theory and L-theory, and
solve those algebraic problems by whatever tools are available. As a result, we
can solve many geometric problems without understanding at all what the
solutions look like.

A beautiful example of this paradoxical state of affairs is the result of
Nabutovsky that despite the result of Smale (proved inter alia in the proof of
the high-dimensional Poincar\'e conjecture) that every smooth codimension one
sphere in the unit $n$-disk ($n>4$) can be isotoped to the boundary, the minimum
complexity of the embeddings required in the course of such an isotopy (measured
by how soon normal exponentials to the embedding intersect) cannot be bounded by
any recursive function of the original complexity of the embedding.
Effectively, an easy isotopy would give such a sphere a certificate of its own
simple connectivity, which is known to be impossible.

In other situations, such as those governed by an $h$-principle, a hard logical
aspect of this sort does not arise.  In this paper we introduce some tools of
quantitative algebraic topology which we hope can be applied to showing that
various geometric problems have solutions of low complexity.

As a first, and, we hope, typical example, we study the problem, emphasized by
Gromov, of trying to understand the work of Thom\footnote{Thom solved the
  unoriented version of this exactly, and only solved the rational version of
  the oriented question. However, later work of Milnor and Wall did the more
  difficult homotopy theory necessary for the oriented case.} 
on cobordism.  Given a closed smooth (perhaps oriented) manifold, the cobordism
question is whether it bounds a compact (oriented) manifold.  The answer to this
is quite checkable: it is determined by whether the cycle represented by the
manifold in the relevant (i.e. $\mathbb{Z}$ or $\mathbb{Z}/2\mathbb{Z}$)
homology of a Grassmannian (where the manifold is mapped in via the Gauss map
classifying the manifold's stable normal bundle) is trivial.

This raises two questions: the first is how the geometry of a manifold is
reflected in the algebraic topological problem, and second is, how difficult
is it to find the nullhomotopy predicted by the algebraic topology.  As a test
of this combined problem, Gromov suggested the following question: Given a
manifold, assume away small scale problems by giving it a Riemannian metric
whose injectivity radius is at least $1$, and whose sectional curvature is
everywhere between $-1$ and $1$.  These properties can be achieved through a
rescaling.  A manifold possessing these properties will be said to have
\emph{bounded local geometry}.  The geometric complexity of such a manifold
can be measured by its volume.

If $M$ is a smooth compact manifold, without a specified metric, we measure its
(differential-)\hspace{1sp}topological complexity by the infimum of the
geometric complexity over all metrics with bounded local geometry.  (If $M$ is
not closed, we require it to look like a collar $\partial M \times [0,1]$ within
distance 1 of the boundary.)  This is a reasonable complexity measure: there are
only finitely many diffeomorphism classes of manifolds with a given bound on
complexity \cite{CheeFin} \cite{Pet} \cite[\S8D]{GrMS}.

The central question is as follows.  Given a smooth (oriented) manifold $M^n$ of
complexity $V$ which is null-cobordant, what is the least complexity of a
null-cobordism?  That is, if $W$ is an (oriented) compact Riemannian
$(n+1)$-manifold of bounded local geometry which bounds a manifold diffeomorphic
to $M$, how small can the volume of $W$ be?  Gromov has observed
\cite[\S 5$\frac{5}{7}$\,II]{GroPC} that tracing through the relevant
mathematics would give a tower of exponentials of $V$ (of size around the
dimension of the manifold minus $2$), but has suggested \cite{GrQHT} that the
truth might be linear.

The linearity problem, if it has an affirmative solution, would require very new
geometric ideas, and seemingly a solution to the cobordism problem essentially
different from Thom's.  We build on Thom's work to obtain the following:
\begin{thmA} \label{filling}
  If $M$ is an (oriented) closed smooth null-cobordant manifold of complexity
  $V$, then it has a null-cobordism of complexity at most
  $$c_1(n) V^{c_2(n)}.$$
\end{thmA}

Unfortunately, the degree of this polynomial obtained by tracing through our
arguments grows exponentially.  Presumably, this can be substantially improved.
F.~Costantino and D.~Thurston have already shown that for 3-manifolds, one does
not need worse than quadratic growth for the complexity of the null-cobordism
\cite{CoTh}.\footnote{Though they use a PL measure of complexity, the number of
  simplices in a triangulation.}

Our proof follows the ideas of Thom quite closely and is based on making those
steps quantitative (if suboptimally) and then getting an a priori estimate on
the size of the most efficient nullhomotopy of a Thom map when the homological
condition holds.

Thom's work starts by embedding $M$ into a sphere (or equivalently Euclidean
space).  This is already an act of violence: one knows that this will
automatically introduce distortion.  This is one source of growth that we don't
know how to avoid.\footnote{A proof of the non-oriented cobordism theorem was
  given by \cite{BH} without using embedding.  However, at a key moment there's
  a ``squaring trick'' in the proof, which also ends up giving, as a result of
  an induction, a polynomial estimate with an $\exp(n^2)$ degree polynomial.} 

For manifolds embedded in the sphere, the Lipschitz constant of the Thom map is
closely related to the complexity of the submanifold\footnote{Thom produces the
  null-cobordism from a nullhomotopy by taking a transverse inverse image.}
and the thickness of a tubular neighborhood.  Conversely, if we know something
about the Lipschitz constant of a nullhomotopy of the Thom map, we can extract a
geometrically bounded transverse inverse image.

Zooming in, we see three issues that need to be taken care of.
\begin{enumerate}
\item We need to bound the Lipschitz constants of the maps at time $t$ in a
  nullhomotopy (its ``thickness''.)  Gromov has suggested \cite{GrQHT} that
  these frequently have a linear bound for maps of finite complexes into finite
  simply connected complexes.\footnote{If the domain is a circle and the target
    is a 2-complex, then for manifolds with unsolvable word problem, there can
    be no computable upper bound for the worst Lipschitz constant in a
    nullhomotopy.  But for many groups with small Dehn function, it is possible
    to do this with only a linear increase.  In particular, simple connectivity
    is an extremely natural requirement.}
\item Bounding the worst Lipschitz constant arising in a nullhomotopy does not
  quite suffice.  One needs to bound the width\footnote{The Lipschitz constant
    in the time direction.} of the nullhomotopy as well.  This is a nontrivial
  issue: a nullhomotopy of thickness $L$ can in general be replaced by one of
  width $\exp(L^d)$ where $d$ is the dimension of the domain, but this is the
  best ``automatic'' bound.
\item Even provided such bounds, a transverse inverse image may be very large
  compared to the original manifold.
\end{enumerate}
We deal with (1) and (2) simultaneously; this is the homotopy-theoretic result
mentioned earlier.  The real loss in our theorem comes from (3).  In order to
find a quantitative embedding of our manifold into $S^N$, we are forced to take
$N$ to be very large, and the embedded submanifold has small support in the
resulting sphere.  However, the support of a nullhomotopy may still be quite
large.  This problem of the increase in the support is also one we have made no
progress on, and which seems important in a broader context than just cobordism
theory.

\subsection{Building Lipschitz homotopies}

The main technical result of the paper is the following:
\begin{thmA} \label{intro:main}
  Let $X$ be an $n$-dimensional finite complex and $Y$ a finite complex which is
  rationally equivalent to a product of Eilenberg--MacLane spaces through
  dimension $n$.  If $f,g:X \rightarrow Y$ are $L$-Lipschitz homotopic maps,
  then there is a homotopy between them which is $C(X,Y)L$-Lipschitz as a map
  from $X \times [0,1]$ to $Y$.
\end{thmA}
The simplest setting in which this theorem applies is when $Y$ is an
odd-dimensional sphere, or when $Y$ is a $2k$-sphere and $n \leq 4k-2$.  More
generally, $Y$ may be any Lie group or, even more generally, H-space.  Given
that the targets in many topological problems are H-spaces, we are optimistic
that this partial result regarding the linearity of homotopies will have more
general application.  (We give an example below showing that this theorem cannot
be extended to arbitrary simply connected complexes in place of $Y$.)

One antecedent to this result is given in \cite{FWPNAS}, where maps with target
possessing finite homotopy groups are studied.  In that setting, the width of a
nullhomotopy is actually bounded universally, independent of $X$.  On the other
hand, that paper shows that for any space with infinite homotopy groups there
cannot be too uniform of an estimate of a linear upper bound on nullhomotopies.

The obstruction in \cite{FWPNAS} has to do ultimately with homological filling
functions.  Isoperimetry likewise comes up in our result, and is best
appreciated by considering the following very concrete setting:
\begin{lem*}
  If $f: S^2 \rightarrow S^2$ is a degree zero map with Lipschitz constant $L$,
  then there is a $CL$-Lipschitz nullhomotopy for some $C$.
\end{lem*}
This can be proved following the classical idea of Brouwer of cancelling point
inverses with opposite local degree, but in a careful layered way so as to be
able to control the Lipschitz constants.  We will give a careful explanation of
this as it provides the main intuition for the proof of Theorem
\ref*{intro:main}.

\subsection{Obstruction theory} \label{S:S2}

Let $f:S^2 \to S^2$ be a nullhomotopic $L$-Lipschitz map.  We assume this has a
very particular structure; later we will see that such a structure can be
obtained with only small penalties on constants.  The domain sphere $X$ is a
subdivision of a tetrahedron into a grid isometric subsimplices, $L$ to a side.
The map $f$ maps its 1-skeleton to the basepoint; for every 2-simplex either it
also maps it to the basepoint, or it maps a ball in the simplex homeomorphically
to $S^2$ minus the basepoint, with degree $\pm 1$.
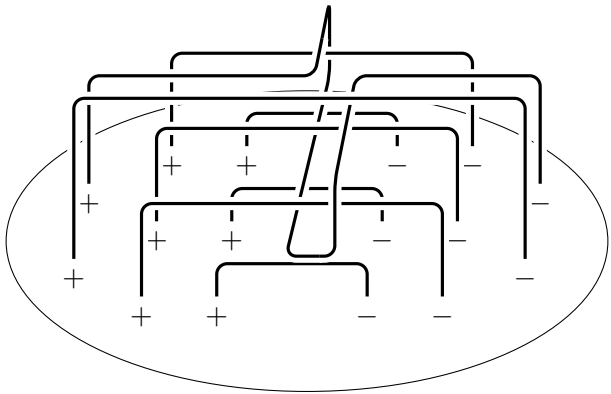
\begin{figure} 
  \centering
  \begin{tikzpicture}
    \tikzstyle{backing}=[white,line width=5pt];
    \draw (0,0) ellipse (4 and 2);
    \node (p1) at (-2.9,0.5) {$+$};
    \node (p2) at (-3.1,-0.5) {$+$};
    \node (p3) at (-1.8,1) {$+$};
    \node (p4) at (-2,0) {$+$};
    \node (p5) at (-2.2,-1) {$+$};
    \node (p6) at (-0.8,1) {$+$};
    \node (p7) at (-1,0) {$+$};
    \node (p8) at (-1.2,-1) {$+$};
    \node (m1) at (3.1,0.5) {$-$};
    \node (m2) at (2.9,-0.5) {$-$};
    \node (m3) at (2.2,1) {$-$};
    \node (m4) at (2,0) {$-$};
    \node (m5) at (1.8,-1) {$-$};
    \node (m6) at (1.2,1) {$-$};
    \node (m7) at (1,0) {$-$};
    \node (m8) at (0.8,-1) {$-$};
    \draw[backing,rounded corners] (p6) -- +(0,0.7) -- +(2,0.7) -- (m6);
    \draw[very thick,rounded corners] (p6) -- +(0,0.7) -- +(2,0.7) -- (m6);
    \draw[backing,rounded corners] (p7) -- +(0,0.7) -- +(2,0.7) -- (m7);
    \draw[very thick,rounded corners] (p7) -- +(0,0.7) -- +(2,0.7) -- (m7);
    \draw[backing,rounded corners] (p8) -- +(0,0.7) -- +(2,0.7) -- (m8);
    \draw[very thick,rounded corners] (p8) -- +(0,0.7) -- +(2,0.7) -- (m8);
    \draw[backing,rounded corners] (0.2,2.7) -- +(0.1,0.5) -- +(0.1,-0.5) --
    +(-0.48,-2.9) -- +(0,-2.9);
    \draw[very thick,rounded corners] (0.2,2.7) -- +(0.1,0.5) -- +(0.1,-0.5) --
    +(-0.48,-2.9) -- +(0,-2.9);
    \draw[backing,rounded corners] (p3) -- +(0,1.5) -- +(4,1.5) -- (m3);
    \draw[very thick,rounded corners] (p3) -- +(0,1.5) -- +(4,1.5) -- (m3);
    \draw[backing,rounded corners] (p4) -- +(0,1.5) -- +(4,1.5) -- (m4);
    \draw[very thick,rounded corners] (p4) -- +(0,1.5) -- +(4,1.5) -- (m4);
    \draw[backing,rounded corners] (p5) -- +(0,1.5) -- +(4,1.5) -- (m5);
    \draw[very thick,rounded corners] (p5) -- +(0,1.5) -- +(4,1.5) -- (m5);
    \draw[backing,rounded corners] (p1) -- +(0,1.7) -- +(3,1.7) -- +(3.1,2.2)
    +(3.07,-0.7) -- +(3.27,-0.7) -- +(3.27,0.3) -- +(3.55,1.7) -- +(6,1.7)
    -- (m1);
    \draw[very thick,rounded corners] (p1) -- +(0,1.7) -- +(3,1.7) -- +(3.1,2.2)
    +(3.07,-0.7) -- +(3.27,-0.7) -- +(3.27,0.3) -- +(3.55,1.7) -- +(6,1.7)
    -- (m1);
    \draw[backing,rounded corners] (p2) -- +(0,2.4) -- +(6,2.4) -- (m2);
    \draw[very thick,rounded corners] (p2) -- +(0,2.4) -- +(6,2.4) -- (m2);
  \end{tikzpicture}
  \caption{Connecting preimages of opposite orientations with tubes: the global
    picture.  Note that the Lipschitz constant of a nullhomotopy depends only on
    the thickness of the tubes; therefore inefficiencies in routing only matter
    insofar as they force many tubes to bunch up in the same region.}
\end{figure}

To construct a nullhomotopy of $f$, we need to connect the positive and negative
preimages with tubes in $X \times [0,1]$.  Care must be taken to route these
tubes in such a way that there are not too many clustered in any given spot.  To
do this, we decide beforehand how many tubes need to go through any particular
part of $X \times [0,1]$ and then connect them up in any available way.

To make this precise, assume that the tubes miss $X^{(0)} \times [0,1]$.  Then we
can count the number of tubes going through $p \times [0,1]$ for each 1-simplex
$p$ of $X$.  Every tube that goes into $q \times [0,1]$, for any 2-simplex $q$,
must either come out through another edge or come back to 0.  In other words, if
$\alpha \in C^1(X;\mathbb{Z})$ is the cochain which indicates the number of
tubes (with sign!) going through $p \times [0,1]$, then $\omega=\delta\alpha$
gives the degree of $f$ on 2-simplices of $X$.  In the language of obstruction
theory, $\omega$ is the obstruction to nullhomotoping $f$, and the existence of
$\alpha$ demonstrates that the obstruction can be resolved.

\begin{figure} 
  \centering
  \begin{subfigure}{0.7\textwidth}
    \centering
  \begin{tikzpicture}[scale=1.3]
    \draw (0,2.5) -- (.2,0) -- (1.4,1.3) -- cycle;
    \node(base) at (.65,1.25){$\langle \omega,q \rangle$};
    \node(top) at (8.5,1.5){$0$};
    \draw (0,2.5) -- (.2,0) -- (1.4,1.3) -- cycle;
    \foreach \x in {1.5,3,4.5,8} {
      \draw (\x,2.5) -- (\x+1.4,1.3) -- (\x+.2,0);
      \draw[dotted] (\x,2.5) -- (\x+.2,0);
    }
    \draw (0,2.5) -- (8,2.5) (.2,0) -- (8.2,0) (1.4,1.3) -- (9.4,1.3);
    \node(etc) at (6.8,1.8){\huge{$\cdots$}};
    \node(etc2) at (7,.7){\huge{$\cdots$}};
    \node(side1) at (1.4,2){
      $\left\lfloor \frac{1}{CL}\langle\alpha,p_0\rangle \right\rfloor$};
    \draw[->] (2.7,2.8) node[anchor=west]{
        $\left\lfloor \frac{2}{CL}\langle\alpha,p_0\rangle \right\rfloor
         -\left\lfloor \frac{1}{CL}\langle\alpha,p_0\rangle \right\rfloor$}
      .. controls (2.3,2.8) and (2.4,2.3) .. (2.9,1.8);
    \draw[->] (3.2,-.4) node[anchor=west]{
        $\langle \omega,q \rangle-\sum_{p \in \partial q}
         \left\lfloor \frac{2}{CL}\langle\alpha,p\rangle \right\rfloor$}
      .. controls (2,-.4) and (4.5,1.2) .. (3.8,1.2);
    \draw[->] (8.7,2.5) node[anchor=south]{edge $p_0$} -- (8.7,2);
  \end{tikzpicture}
  \caption{Degrees on 2-cells of prisms.}
  \end{subfigure}
  \begin{subfigure}[b]{0.7\textwidth}
    \centering
  \begin{tikzpicture}[scale=1.3]
    \draw (0,2.5) -- (.2,0) -- (1.4,1.3) -- cycle;
    \draw (0,2.5) -- (.2,0) -- (1.4,1.3) -- cycle;
    \foreach \x in {1.5,3,4.5,8} {
      \draw (\x,2.5) -- (\x+1.4,1.3) -- (\x+.2,0);
      \draw[dotted] (\x,2.5) -- (\x+.2,0);
    }
    \draw (0,2.5) -- (8,2.5) (.2,0) -- (8.2,0) (1.4,1.3) -- (9.4,1.3);
    \node(etc) at (6.8,1.8){\huge{$\cdots$}};
    \node(etc2) at (7,.7){\huge{$\cdots$}};
    \draw (.6,1.25) circle [x radius = .25, y radius = .3] node{$+$};
    \draw (.6,1.55) .. controls (1.2,1.55) .. (1.14,2);
    \draw (.6,.95) .. controls (1.6,.95) .. (1.66,1.8);
    \draw (1.4,1.9) circle [x radius = .2, y radius = .3, rotate=45] node{$-$};
    \draw (3,.7) circle [x radius = .2, y radius = .3, rotate=-45] node{$+$};
    \draw (5.1,1.55) -- (3.6,1.55) .. controls (2.74,1.55) .. (2.74,.6);
    \draw (5.1,.95) -- (3.6,.95) .. controls (3.26,.95) .. (3.26,.8);
    \draw (3.6,1.25) circle [x radius = .25, y radius = .3];
    \draw (5.1,1.25) circle [x radius = .25, y radius = .3] node{$-$};
    \draw (2.4,2.15) circle [x radius = .2, y radius = .3, rotate=45] node{$-$};
    \draw (2.14,.2) -- (2.14,2.3) (2.66,.45) -- (2.66,2.05);
    \draw (2.4,.35) circle [x radius = .2, y radius = .3, rotate=-45] node{$+$};
  \end{tikzpicture}
    \caption{Connecting homeomorphic preimages of $S^2 \setminus *$ with tubes.}
  \end{subfigure}
  \caption{Constructing a nullhomotopy: the local picture.}
\end{figure}
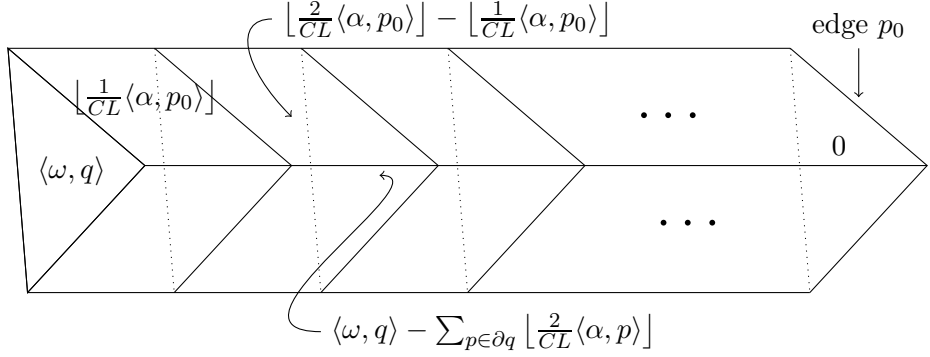
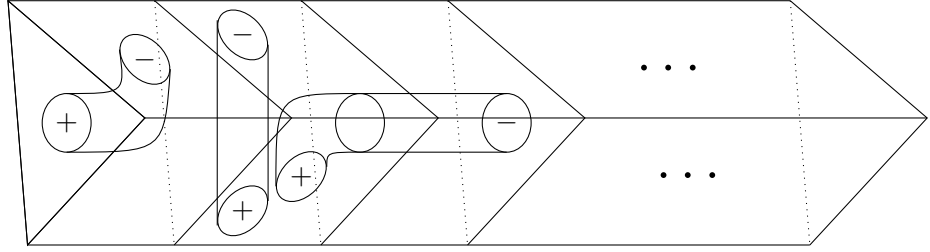
To ensure that it can be resolved efficiently, we need to pick a relatively
small $\alpha$.  The best we can do is to choose an $\alpha$ which takes values
$\leq CL$.  By considering a situation with degree $O(L^2)$ on one side of $X$
canceling out degree $-O(L^2)$ on the other side, we see that we can do no
better.  That this is also the worst possible situation follows from the
classical isoperimetric inequality for spheres; this is discussed in much
greater generality in Section \ref{Siso}.

In effect, once we have set $\alpha$, deciding how many tubes must go through a
given point, we can connect them up in an entirely local way.  We give
$X \times [0,1]$ a cellulation by prisms of length $1/CL$ and base the
2-simplices of $X$.  We then construct the map $F$ by skeleta on this
cellulation:
\begin{enumerate}
\item First, map the 1-skeleton to the basepoint.
\item Next, we can map the 2-cells via maps of degree between $-3$ and $3$ in
  such a way that the map on the boundary of each prism has total degree zero.
  (It is here that we ``layer'' the nullhomotopy.)
\item Finally, we choose a way to connect pairs of preimages on each prism via
  tubes.  Since the number of tubes in each prism is bounded, we can do this
  with bounded Lipschitz constant.
\end{enumerate}
For the second step, we need to use our $\alpha$.  If we ensure that for each
1-simplex $p$ of $X$, the degree of $F$ on $p \times [0,1]$ is
$\langle\alpha,p\rangle$, then $F$ will have degree 0 on the boundary of each
``long prism'' $q \times [0,1]$, where $q$ is a 2-simplex of $X$.

It remains to make sure that the degree is zero on the ``short prisms''.  To do
this, we spread $\langle\alpha,p\rangle$ as evenly as possible along the unit
interval: for every integer $1 \leq t \leq CL$, the degree of $\alpha$ on
$p \times [0,t/CL]$ is $\lfloor \frac{t}{CL}\langle\alpha,p\rangle \rfloor$.
This then also determines the required degree on $q \times \{t/CL\}$ for every
2-simplex $q$ and time $t$ to make the total degree on the boundary of each
prism zero.  It is easy to check that the resulting degrees on all 2-cells are
at most 3.

\subsection{Outline of proof of Theorem \ref*{intro:main}}

We now describe how the proof of the above Lemma leads to the proof of Theorem
\ref*{intro:main}.  The motto is the same: if we can kill the obstruction to
finding a homotopy, then we can do the killing in a bounded way.

The first step is to reduce to a case where obstruction theory applies.  For
this, we simplicially approximate our map in a quantitative way.  That is, given
a map $X \to Y$ between metric simplicial complexes, the fineness of the
subdivision of $X$ must be inversely proportional to the Lipschitz constant of
the map.

From here, the general strategy is to build a homotopy by induction on the
skeleta of $X \times I$ with a product cell structure.  This homotopy will not
in general be simplicial, but it will have the property that restrictions to
each cell form a fixed finite set depending only on $X$ and $Y$.  Every time we
run into a null-cohomologous obstruction cocycle, we use a cochain that it
bounds to modify the map on the previous skeleton.  We ensure that these
modifications are chosen from a fixed finite set of maps, leaving us with a
fixed finite set of maps on the boundaries of cells one dimension higher.  Then
we can fill each such map in a fixed way, preserving the desired property.

When the obstructions are torsion, the main issue is the well-known one that
killing obstructions ``blindly'' will sometimes lead to a dead end even when a
homotopy exists.  On the other hand, since there is a finite number of choices
of torsion values for a cochain to take, we may avoid this by following a ``road
map'' given by a known, but potentially uncontrolled, nullhomotopy of $f$.  This
is the content of Lemma \ref{lem:Qlift}.

On the other hand, when we get integral obstructions, our choice of rational
homotopy structure ensures that such issues do not come up.  On the other hand,
we do need to worry about isoperimetry.  This is covered by Theorem
\ref{thm:ooo}, which generalizes the argument above.

\subsection{Acknowledgments}
This paper owes a lot to the ideas of Gromov, as this introduction makes amply
clear.  The last author would like to thank Steve Ferry for a collaboration that
began this work.  Essentially, the polynomial bound in the non-oriented case can
be obtained by combining \cite{FWPNAS} with some of the embedding arguments in
this paper.  We also thank MSRI for its hospitality during a semester (long ago)
when we began working towards the results reported here.  The authors would also
like to thank Alexander Nabutovsky and Vitali Kapovitch for pointing out
simplifications to several proofs, and for many useful conversations.  Finally,
we would like to thank the anonymous referee for a large number of remarkably
insightful suggestions and corrections.  The first author was partially
supported by an NSERC postdoctoral fellowship.

\section{Preliminaries}

In this section, we discuss how to subdivide a metric simplicial complex so that
the edges all have length approximately $1/L$ for a specified $L$.  We also show
that, for any simplicial map $f: X \rightarrow Y$ and any $L$, we can subdivide
$X$ as above to form $X_L$ and homotope $f$ through a short homotopy
to $\widetilde{f}: X_L \rightarrow Y$.

\subsection{Regular subdivision of simplices}
\begin{defn}
  Define a \emph{simplicial subdivision scheme} to be a family, for every $n$
  and $L$, of metric simplicial complexes $\Delta^n(L)$ isometric to the
  standard $\Delta^n$ with length 1 edges, such that $\Delta^n(L)$ restricts to
  $\Delta^{n-1}(L)$ on all faces.  A subdivision scheme is \emph{regular} if for
  each $n$ there is a constant $A_n$ such that $\Delta^n(L)$ has at most $A_n$
  isometry classes of simplices and a constant $r_n$ such that all 1-simplices
  of $\Delta^n(L)$ have length in $[r_n^{-1}L^{-1},r_nL^{-1}]$.

  Given a regular subdivision scheme, we can define the
  \emph{$L$-regular subdivision} of any metric simplicial complex, where each
  simplex is replaced by an appropriately scaled copy of $\Delta^n(L)$.
\end{defn}
Note that $L$ times barycentric subdivision is \emph{not} regular.  On the other
hand, there are at least two known examples of regular subdivision.  One is the
edgewise subdivision of Edelsbrunner and Grayson \cite{EdGr}, which has the
advantages that the $L$-regular subdivision of $\Delta^n(M)$ is $\Delta^n(LM)$
and that the lengths of edges vary by a factor of only $\sqrt{2}$.  Roughly, the
method is to cut the simplex into small polyhedra by planes parallel to the
$(n-1)$-dimensional faces, then partition each such polyhedron into simplices in
a standard way.  The other is described by Ferry and Weinberger \cite{FWPNAS}:
the trick is to subdivide $\Delta^n$ into $n+1$ identical cubes, then subdivide
these in the obvious way into $L^n$ cubes, and finally subdivide these in a
canonical way into simplices.  This method has the advantage of being easy to
describe.

None of the listed advantages is crucial for our continued discussion, so we may
remain agnostic as to how precisely we subdivide our simplices.

\subsection{Simplicial approximation}
\begin{prop}[Quantitative simplicial approximation theorem]
  \label{prop:quantitative_simplicial_approximation}
  For finite simplicial complexes $X$ and $Y$ with piecewise linear metrics,
  there are constants $C$ and $C^\prime$ such that any $L$-Lipschitz map
  $f:X \to Y$ has a $CL$-Lipschitz simplicial approximation via a
  $(CL+C^\prime)$-Lipschitz homotopy.
\end{prop}
\begin{proof}
  We trace constants through the usual proof of the simplicial approximation
  theorem, as given in \cite{Hatc}.

  Denote the open star of a vertex $v$ by $\st v$.  Let $c$ be a Lebesgue number
  for the open cover $\{\st w \mid w$ is a vertex of $Y\}$ of $Y$, that is, a
  number such that every $c$-ball in $Y$ is contained in one of the sets in the
  cover.  Then $c/L$ is a Lebesgue number for the open cover $\{f^{-1}(\st w)\}$
  of $X$.  Take a regular subdivision $X_L$ of $X$ so that for some $0<d(X)<1/2$
  each simplex of $X_L$ has diameter between $dc/L$ and $c/2L$.  Hence $f$ maps
  the closed star of each vertex $v$ of $X_L$ to the open star of some vertex
  $g(v)$ of $Y$.  This gives us a map $g:X_L^{(0)} \to Y^{(0)}$ which takes
  adjacent vertices of $X_L$ to adjacent vertices of $Y$, and hence if $\ell$ is
  the maximum edge length of $Y$, $g$ is $\ell L/dc$-Lipschitz.

  By a standard argument, this map $g$ extends linearly to a map $g:X_L \to Y$
  with the same Lipschitz constant.  The linear homotopy from $f$ to $g$ has
  Lipschitz constant $\max\{\ell L/dc, \ell\}$.
\end{proof}
\begin{rmk}
  Suppose that $Y$ and $X$ are $n$-dimensional and made up of standard simplices
  of edge length 1.  Then $c$ is the inradius of a standard simplex,
  $c=\frac{1}{\sqrt{2n(n+1)}}$, and by using the edgewise subdivision we can
  make sure that $d>1/2\sqrt{2}$.  Thus the Lipschitz constant of the map
  increases by a factor of at most
  $$C \leq 4\sqrt{n(n+1)}.$$
  Furthermore, if $X$ is 2-dimensional, then all of the edge lengths of the
  subdivision are equal.  Therefore, in this case, $C \leq 4\sqrt{3}$, and in
  fact approaches $2\sqrt{3}$ for large $L$, since we can choose a subdivision
  parameter very close to $L$, and thus $d$ very close to $1$.
\end{rmk}
We will use simplicial approximation mainly as a way of ensuring that our maps
have a uniformly finite number of possible restrictions to simplices.  Almost
all instances of ``simplicial'' in this paper can be replaced with ``such that
the restrictions to simplices are chosen from a finite set associated with the
target space.''  This formulation makes sense even when the target space is not
a simplicial complex.  In particular, it is preserved by postcomposition with
any map, for example one collapsing certain simplices.

\section{Isoperimetry for integral cochains} \label{Siso}

The goal of this section is to prove the following (co)isoperimetric inequality.
\begin{lem}[$\ell^\infty$ coisoperimetry] \label{lem0'}
  Let $X$ be a finite simplicial complex equipped with the standard metric, and
  let $X_L$ be the cubical or edgewise $L$-regular subdivision of $X$, and $k
  \geq 1$.  Then there is a constant $C_{\mathrm{IP}}=C_{\mathrm{IP}}(X,k)$ such that
  for any simplicial coboundary $\omega \in C^k(X_L;\mathbb{Z})$, there is an
  $\alpha \in C^{k-1}(X_L;\mathbb{Z})$ with $d\alpha=\omega$ such that
  $\lVert\alpha\rVert_\infty \leq C_{\mathrm{IP}}L\lVert\omega\rVert_\infty$.
\end{lem}
We will start by proving the much easier version over a field; in the rest of
the section $\mathbb{F}$ will denote $\mathbb{Q}$ or $\mathbb{R}$.  Then we will
demonstrate how to find an integral filling cochain near a rational or real one.
\begin{lem} \label{lem0F}
  Let $X$ be a finite simplicial complex equipped with the standard metric, and
  let $X_L$ be an $L$-regular subdivision of $X$.  Then for any $k$, there is a
  constant $K=K(X,k)$ such that for any simplicial coboundary $\omega \in
  C^k(X_L;\mathbb{F})$, there is an $\alpha \in C^{k-1}(X_L;\mathbb{F})$ with
  $d\alpha=\omega$ such that $\lVert\alpha\rVert_\infty \leq
  KL\lVert\omega\rVert_\infty$.
\end{lem}
\begin{proof}
  We first show a similar isoperimetric inequality, and then demonstrate that it
  is equivalent to the coisoperimetric version.
  \begin{lem} \label{lem1}
    There is a constant $K=K(X,k)$ such that boundaries $b \in
    C_{k-1}(X_L;\mathbb{F})$ of simplicial volume $V$ bound chains of simplicial
    volume at most $KLV$.
  \end{lem}
  \begin{proof}
    There are two ways we can measure the volume of a simplicial $i$-chain in
    $X_L$.  The first, \emph{simplicial volume}, is given by assigning every
    simplex volume 1, i.e.
    $$\vol\left(\sum \alpha_ip_i\right)=\sum \lvert\alpha_i\rvert.$$
    Alternatively, we can measure the $i$-\emph{mass} of chains: the mass of a
    simplex $p$ is its Riemannian $i$-volume, and in general
    $$\mass\left(\sum \alpha_ip_i\right)=\sum \lvert\alpha_i\rvert\vol_i(p_i).$$
    Thus there are constants $K_i$ and $K^\prime_i$, depending on the choice of
    subdivision scheme, such that for every $i$-chain $c$,
    $$K_iL^i\mass c \leq \vol c \leq K^\prime_iL^i.$$
    Therefore to prove the lemma it suffices to show that a boundary whose
    $(k-1)$-mass in $X$ is $V$ bounds a chain whose $k$-mass is at most $KV$.

    Our main tool here is the Federer--Fleming deformation theorem, a powerful
    result in geometric measure theory which allows very general chains to be
    deformed to simplicial ones in a controlled way.  One proves this result by
    shining a light from the right spot inside each simplex so that the
    resulting shadow on the boundary of the simplex is not too large.  By
    iterating this procedure on simplices of each dimension between $n$ and
    $k+1$, we eventually end up with a shadow in the $k$-skeleton, which is the
    desired simplicial chain.  Federer and Fleming's original version
    \cite[Thm.~5.5]{FF} was based on deformation to the standard cubical lattice
    in $\mathbb{R}^n$.  However, everything in their proof, except for the
    precise constants, translates to simplicial complexes.  (See
    \cite[Thm.~10.3.3]{ECHLP} for a proof of a slightly narrower analogue in the
    case of triangulated manifolds, which however also applies to any simplicial
    complex.)

    Federer and Fleming's theorem works for normal currents.  To avoid this
    rather technical concept, we state the result for Lipschitz chains, that is,
    singular chains whose simplices are Lipschitz.
    \begin{thm*}[Federer--Fleming deformation theorem]
      Let $W$ be an $n$-dimensional simplicial complex with the standard metric
      on each simplex.  There is a constant $\rho(k,n)$ such that the following
      holds.  Let $T$ be a Lipschitz $k$-chain in $W$ with coefficients in
      $\mathbb{F}$.  Then we can write $T=P+Q+\partial S$, where
      \begin{enumerate}
      \item $\mass(P) \leq \rho(k,n)(\mass(T)+\mass(\partial T))$;
      \item $\mass(Q) \leq \rho(k,n)\mass(\partial T)$;
      \item $\mass(S) \leq \rho(k,n)\mass(T)$;
      \item $P$ can be expressed as an $\mathbb{F}$-linear combination of
        $k$-simplices of $W$.
      \item If $\partial T$ can already be expressed as a combination of
        $(k-1)$-simplices of $W$ (for example, if $T$ is a cycle), then $Q=0$
        and
        $$\mass(P) \leq \rho(k,n)(\mass(T).$$
      \end{enumerate}
    \end{thm*}
    Now suppose that $W$ is given a metric $d_W$ whose simplices are not
    standard, but such that the identity map
    $\iota:(W,d_{\mathrm{std}}) \to (W,d_W)$ satisfies
    $$\lambda_1d(x,y) \leq d(\iota(x),\iota(y)) \leq \lambda_2d(x,y)$$
    for all $x,y \in W$.  When mass is measured with respect to $d_W$, the
    bounds in the theorem become
    \begin{enumerate}
    \item $\displaystyle{\mass(P) \leq
      \rho(k,n)\left(\frac{\lambda_2^k}{\lambda_1^k}\mass(T)
      +\frac{\lambda_2^{k-1}}{\lambda_1^k}\mass(\partial T)\right)}$;
    \item $\displaystyle{\mass(Q) \leq
      \frac{\lambda_2^{k-1}}{\lambda_1^k}\rho(k,n)\mass(\partial T)}$;
    \item $\displaystyle{\mass(S) \leq
      \frac{\lambda_2^{k+1}}{\lambda_1^k}\rho(k,n)\mass(T)}$.
    \end{enumerate}
    We apply the theorem twice.  First, we apply it to $b$ as a Lipschitz cycle
    in $X$, to show that it is homologous to a $(k-1)$-cycle $P \in
    C_{k-1}(X;\mathbb{F})$ of volume $\leq C(k)V$ via a Lipschitz $k$-chain $S$
    of volume $\leq C(k)V$.  Next, we apply it to $S$ as a $k$-chain in $X_L$.
    Notice that the ratio $\lambda_2/\lambda_1$ is bounded independent of $L$
    for a regular subdivision; therefore, $S$ deforms rel boundary to a chain in
    $C_k(X_L;\mathbb{F})$ of volume $\leq C(k)C^\prime(k)V$, where $C^\prime$
    depends on the subdivision scheme.  Finally, $P$ bounds a chain in $X$ of
    volume $\leq C^\dagger C(k)V$, where $C^\dagger$ depends only on the geometry
    of $X$.  Thus we can set $K=(C^\dagger+C^\prime)C$.
  \end{proof}
  \begin{lem}
    Let $X$ be a finite simplicial complex.  Then the following are equivalent
    for any constant $C$:
    \begin{enumerate}
    \item any boundary $\sigma \in C_{k-1}(X;\mathbb{F})$ has a filling $\tau
      \in C_k(X;\mathbb{F})$ with $\vol\tau \leq C\vol\sigma$;
    \item any coboundary $\omega \in C^k(X;\mathbb{F})$ is the coboundary of
      some $\alpha \in C^{k-1}(X;\mathbb{F})$ with $\lVert\alpha\rVert_\infty
      \leq C\lVert\omega\rVert_\infty$.
    \end{enumerate}
  \end{lem}
  The authors would like to thank Alexander Nabutovsky and Vitali Kapovitch for
  pointing out this simplified proof.
  \begin{proof}
    The cochain complex is dual to the chain complex, and the $L_\infty$-norm on
    cochains is dual to the volume norm on chains.  So consider the general
    situation of a linear transformation between two normed vector spaces
    $T:(V,\lVert\cdot\rVert_V) \to (W,\lVert\cdot\rVert_W)$, and let $C(T)$ be
    the operator norm of the transformation
    $$\bar T^{-1}: (\img T,\lVert\cdot\rVert_W) \to
    (V/\ker T,\lVert\cdot\rVert_{\bar V}),$$
    where the norm of an equivalence class $\bar v \in V/\ker T$ is given by
    $\lVert \bar v \rVert_{\bar V}=\min_{v \in \bar v} \lVert v \rVert_V$.  When
    $T$ is the boundary operator on $C_k(X;\mathbb{F})$, $C(T)$ is exactly the
    minimal constant $C$ in condition (1).  Hence this is also the operator norm
    of the dual transformation $(\bar T^{-1})^*:\img T^* \to W^*/\ker T^*$.  It
    remains to investigate the dual norms on these spaces.

    By the Hahn--Banach theorem, any operator on $\img T$ extends to an operator
    of the same norm on all of $W$.  Hence the dual norm of
    $\lVert\cdot\rVert_W|_{\img T}$ is exactly the norm
    $\lVert \overline{\ph} \rVert_{\overline{W^*}}=\min_{\ph \in \overline{\ph}}
    \lVert \ph \rVert_{W^*}$ on $W^*/\ker T^*$, and similarly the dual norm of
    $\lVert\cdot\rVert_{\bar V}$ is $\lVert\cdot\rVert_{V^*}|_{\img T^*}$.
    Therefore the operator norm of $(\bar T^{-1})^*$ is the minimal constant of
    condition (2).
  \end{proof}
  Putting these together, we complete the proof of the rational and real
  coisoperimetry lemma.
\end{proof}
Now we introduce the ingredients for proving the integral version.
\begin{defn}
  A \emph{$k$-spanning tree} of a simplicial complex $X$ is a $k$-dimensional
  subcomplex $T$ which contains $X^{(k-1)}$, such that the induced map
  $H_{k-1}(T;\mathbb{Q}) \to H_{k-1}(X;\mathbb{Q})$ is an isomorphism and
  $H_k(T;\mathbb{Q})=0$.  A \emph{$k$-wrapping tree} of $X$ is a $k$-dimensional
  subcomplex $U$ which contains $X^{(k-1)}$ and such that $H_{k-1}(U;\mathbb{Q})
  \to H_{k-1}(X;\mathbb{Q})$ and $H_k(U;\mathbb{Q}) \to H_k(X;\mathbb{Q})$ are
  both isomorphisms.
\end{defn}
\begin{lem}
  Every simplicial complex $X$ has a $k$-spanning tree and a $k$-wrapping tree.
\end{lem}
\begin{proof}
  A $k$-spanning tree for any $X$ can be constucted greedily starting from
  $X^{(k-1)}$.  At each step, we find a $k$-simplex $c$ in $X$ such that
  $\partial c$ represents a nonzero class in $H_{k-1}(T;\mathbb{Q})$ and add it
  to $T$.  Once there are no such simplices left, $H_{k-1}(T;\mathbb{Q}) \to
  H_{k-1}(X;\mathbb{Q})$ is an isomorphism.  By construction, $T$ has no rational
  $k$-cycles.

  Notice that every $k$-simplex of $X$ outside $T$ is a cycle in
  $C_k(X,T;\mathbb{Q})$.  To build a $k$-wrapping tree from a $k$-spanning tree,
  we may choose a basis for $H_k(X,T;\mathbb{Q})$ from among the simplices and
  add it to the tree.
\end{proof}
Informally speaking, a $k$-spanning tree should be thought of as the least
subcomplex $T$ so that every $k$-simplex outside $T$ is a cycle mod $T$; a
$k$-wrapping tree is the least subcomplex $U$ so that every $k$-simplex outside
$U$ is a boundary mod $U$.  In both cases, the minimality means that there is a
unique ``completion'' for a $k$-simplex $q$, i.e.~a chain $c$ supported in $T$
(respectively, $U$) so that $c+q$ is a cycle (resp., boundary.)

Such spanning trees have been previously studied by Kalai \cite{Kalai} and
Duval--Klivans--Martin \cite{DKM} and \cite{DKM2} in the case where $k$ is the
dimension of the complex.  In that case, the $k$-simplices not contained in a
spanning tree $T$ form a basis for $H_k(X,T;\mathbb{Q})$ (and a $k$-wrapping
tree is simply the whole complex.)  When $X$ contains simplices in dimension
$k+1$, however, there may be relations between the simplices when viewed as
cycles in $X$ modulo $T$.  The next definition attempts to quantify the extent
to which such relations constrain the behavior of cocycles in the pair $(X,T)$.
\begin{defn}
  Let $T$ be a $k$-spanning tree of $X$.  Consider the set $\mathcal{A}$ of
  vectors in $H_k(X,T;\mathbb{Q})$ which are images of $k$-simplices of $X$.  We
  define the \emph{gnarledness}
  $$G(T)=\min \left\{\max_{a \in \mathcal{A}} \lVert a \rVert_1: \text{bases
    $\mathcal{B}$ for $H_k(X,T;\mathbb{Q})$ such that }\mathcal{A} \subset
  \mathbb{Z}\mathcal{B}\right\}.$$
  We say that $T$ is \emph{$G(T)$-gnarled}; we say a basis is \emph{optimal} if
  $\max_{a \in \mathcal{A}} \lVert a \rVert_1$ is minimal in it.
\end{defn}
The gnarledness measures the extent to which certain simplices are homologically
``larger'' than others.  For example, consider a 2-dimensional simplicial
complex which is homeomorphic to the mapping telescope of a degree two self-map
of $S^1$,
$$X=S^1 \times [0,1]/(x,1) \sim (-x,1).$$
Let's say we take a one-dimensional spanning tree $T$ which includes all but one
of the simplices of both $S^1 \times \{0\}$ and $S^1 \times \{1\}$; let $e_0$
and $e_1$, respectively, be the relevant 1-simplices in $X \setminus T$.  Then
in $H_1(X,T;\mathbb{Q}) \cong \mathbb{Q}$, $[e_0]=2$ and $[e_1]=1$.  For any
basis for $H_1(X,T;\mathbb{Q})$ in which $e_1$ is a lattice point,
$\lVert e_0 \rVert_1 \geq 2$, so the tree $T$ is at least 2-gnarled.  Indeed,
the same will happen for any spanning tree of this complex.
\begin{lem} \label{lem:sub}
  The cubical and edgewise $L$-regular subdivisions of $X$ both admit
  $k$-spanning trees which are at most $C(X)$-gnarled; the gnarledness is
  bounded independent of $L$.
\end{lem}
We will actually show this for grids in a cube complex.  It is routine to modify
this proof to work for the cubical subdivision of a simplicial complex; a
similar construction works for the edgewise subdivision, since it consists of a
grid of subspaces parallel to the faces which is then subdivided in a fixed way
depending on dimension.

We first show that the subdivision of a cube has a ``$k$-spanning tree rel
boundary'' with good geometric properties.  To be precise:
\begin{lem} \label{lem:cube_tree}
  Let $K=I^n$ be cubulated by a grid of side length $1/r$, and let $k \leq n$.
  We refer to
  \begin{itemize}
  \item \emph{cells}, i.e.~faces of the cubulation;
  \item \emph{faces}, i.e.~subcomplexes corresponding to faces of the unit cube;
  \item and \emph{boxes}, i.e.~subcomplexes which are products of subintervals.
  \end{itemize}
  Then there is a $k$-subcomplex $T \subset K$ which contains $K^{(k-1)}$ with
  the following properties:
  \begin{enumerate}
  \item $T \cap \partial K=(\partial K)^{(k-1)}$.
  \item $T$ deformation retracts to $(\partial K)^{(k-1)}$.
  \item Every $k$-cell of $K \setminus T$ is homologous rel $T$ to a chain in
    $\partial K$ whose intersection with each $(n-1)$-face is a box.
  \item More generally, every $k$-dimensional box in $K$ is homologous rel $T$
    to a chain in $\partial K$ whose intersection with each $(n-1)$-face is a
    box.
  \end{enumerate}
\end{lem}
Suppose now that we equip every face of $K$ in dimensions $k \leq i \leq n$ with
subcomplexes satisfying these properties, and let $T^\prime$ be the union of all
these subcomplexes.  Then by induction using property (4), any $k$-cell of
$K \setminus T^\prime$ is homologous rel $T^\prime$ to a union of at most
$2^{n-k}n!/k!$ boxes in the $k$-faces of $K$.  In turn, by property (2), each of
these boxes has at most one cell outside $T^\prime$.  Therefore any $k$-cell is
homologous rel $T^\prime$ to a sum of at most $2^{n-k}n!/k!$ cells in the
$k$-faces.  This is the property that we use to prove Lemma \ref{lem:sub}.
\begin{proof}
  \begin{figure} 
    \centering
    \begin{tikzpicture}
      \tikzstyle{chainfill}=[fill=gray!40];
      \tikzstyle{dott}=[circle,fill=black,scale=0.4];
      \foreach \x in {0,1,2,3} {
        \draw[dashed] (0,\x+1) -- (3,\x+1);
        \draw[dashed] (\x,1) -- (\x,4);
        \node[dott] (l\x) at (0,\x+1) {};
        \node[dott] (r\x) at (3,\x+1) {};
      }
      \node[dott] (b1) at (1,1) {};
      \node[dott] (b2) at (2,1) {};
      \node[dott] (t1) at (1,4) {};
      \node[dott] (t2) at (2,4) {};
      \node[dott] (m1) at (1,2) {};
      \node[dott] (m2) at (1,3) {};
      \draw[very thick] (m1) -- (r1) (m2) -- (r2);
      \newcommand{\oo}{7.7}
      \newcommand{\ya}{-0.9}
      \newcommand{\yb}{0.3}
      \newcommand{\xa}{0.5}
      \newcommand{\xb}{0.55}
      \newcommand{\za}{0}
      \newcommand{\zb}{0.8}
      \foreach \x in {0,1,2,3} {
        \draw (\oo+\x * \xa+3*\ya+3*\za,\x * \xb+3*\yb+3*\zb) --
        (\oo+\x * \xa+3*\ya,\x * \xb+3*\yb) -- (\oo+\x * \xa,\x * \xb);
      }
      \draw (\oo,0) -- (\oo+\xa * 3,\xb * 3) --
      (\oo+\xa * 3+\ya * 3,\xb * 3+\yb * 3) --
      (\oo+\ya * 3,\yb * 3) -- cycle;
      \filldraw[chainfill,thick] (\oo+\xa+2*\ya,\xb+2*\yb) --
      (\oo+3*\xa+2*\ya,3*\xb+2*\yb) -- (\oo+3*\xa+2*\ya+\za,3*\xb+2*\yb+\zb) --
      (\oo+\xa+2*\ya+\za,\xb+2*\yb+\zb) -- cycle;
      \filldraw[chainfill,thick] (\oo+\xa+\ya,\xb+\yb) --
      (\oo+3*\xa+\ya,3*\xb+\yb) -- (\oo+3*\xa+\ya+\za,3*\xb+\yb+\zb) --
      (\oo+\xa+\ya+\za,\xb+\yb+\zb) -- cycle;
      \draw  (\oo+\ya+\za,\zb+\yb) -- (\oo+\za,\zb) -- (\oo+\xa+\za,\xb+\zb);
      \filldraw[chainfill,thick] (\oo+\ya+\za,\zb+\yb) --
      (\oo+\xa+\ya+\za,\zb+\xb+\yb) --
      (\oo+\xa+\za,\xb+\zb) -- (\oo+\xa * 3+\za,\zb+\xb * 3) --
      (\oo+\xa * 3+\ya * 3+\za,\zb+\xb * 3+\yb * 3) --
      (\oo+\ya * 3+\za,\zb+\yb * 3) -- cycle;
      \filldraw[chainfill,thick] (\oo+\xa+2*\ya+\za,\xb+2*\yb+\zb) --
      (\oo+3*\xa+2*\ya+\za,3*\xb+2*\yb+\zb) --
      (\oo+3*\xa+2*\ya+2*\za,3*\xb+2*\yb+2*\zb) --
      (\oo+\xa+2*\ya+2*\za,\xb+2*\yb+2*\zb) -- cycle;
      \filldraw[chainfill,thick] (\oo+\xa+\ya+\za,\xb+\yb+\zb) --
      (\oo+3*\xa+\ya+\za,3*\xb+\yb+\zb) --
      (\oo+3*\xa+\ya+2*\za,3*\xb+\yb+2*\zb) --
      (\oo+\xa+\ya+2*\za,\xb+\yb+2*\zb) -- cycle;
      \draw  (\oo+\ya+2*\za,2*\zb+\yb) -- (\oo+2*\za,2*\zb) --
      (\oo+\xa+2*\za,\xb+2*\zb);
      \filldraw[chainfill,thick] (\oo+\ya+2*\za,2*\zb+\yb) --
      (\oo+\xa+\ya+2*\za,2*\zb+\xb+\yb) --
      (\oo+\xa+2*\za,\xb+2*\zb) -- (\oo+\xa * 3+2*\za,2*\zb+\xb * 3) --
      (\oo+\xa * 3+\ya * 3+2*\za,2*\zb+\xb * 3+\yb * 3) --
      (\oo+\ya * 3+2*\za,2*\zb+\yb * 3) -- cycle;
      \filldraw[chainfill,thick] (\oo+\xa+2*\ya+2*\za,\xb+2*\yb+2*\zb) --
      (\oo+3*\xa+2*\ya+2*\za,3*\xb+2*\yb+2*\zb) --
      (\oo+3*\xa+2*\ya+3*\za,3*\xb+2*\yb+3*\zb) --
      (\oo+\xa+2*\ya+3*\za,\xb+2*\yb+3*\zb) -- cycle;
      \filldraw[chainfill,thick] (\oo+\xa+\ya+2*\za,\xb+\yb+2*\zb) --
      (\oo+3*\xa+\ya+2*\za,3*\xb+\yb+2*\zb) --
      (\oo+3*\xa+\ya+3*\za,3*\xb+\yb+3*\zb) -- (\oo+\xa+\ya+3*\za,\xb+\yb+3*\zb)
      -- cycle;
      \draw (\oo+\za * 3,\zb * 3) --
      (\oo+\xa * 3+\za * 3,\zb * 3+\xb * 3) --
      (\oo+\xa * 3+\ya * 3+\za * 3,\zb * 3+\xb * 3+\yb * 3) --
      (\oo+\ya * 3+\za * 3,\zb * 3+\yb * 3) -- cycle;
      \foreach \x in {0,1,2,3} {
        \draw (\oo+\x * \xa+3*\ya+3*\za,\x * \xb+3*\yb+3*\zb) --
        (\oo+\x * \xa+3*\za,\x * \xb+3*\zb) -- (\oo+\x * \xa,\x * \xb);
      }
      \draw (\oo+\xa+\ya+3*\za,\xb+\yb+3*\zb) -- (\oo+\ya+3*\za,\yb+3*\zb);
      \draw (\oo+\xa+2*\ya+3*\za,\xb+2*\yb+3*\zb) --
      (\oo+2*\ya+3*\za,2*\yb+3*\zb);
      \draw (\oo+\ya+3*\za,\yb+3*\zb) -- (\oo+\ya,\yb) --(\oo+\xa+\ya,\xb+\yb);
      \draw (\oo+2*\ya+3*\za,2*\yb+3*\zb) -- (\oo+2*\ya,2*\yb) --
      (\oo+\xa+2*\ya,\xb+2*\yb);
    \end{tikzpicture}
    
    \caption{
      An illustration of the subcomplex $T$ for $n=2$, $k=1$ and $n=3$, $k=2$.
    } \label{fig:cycle}
  \end{figure}
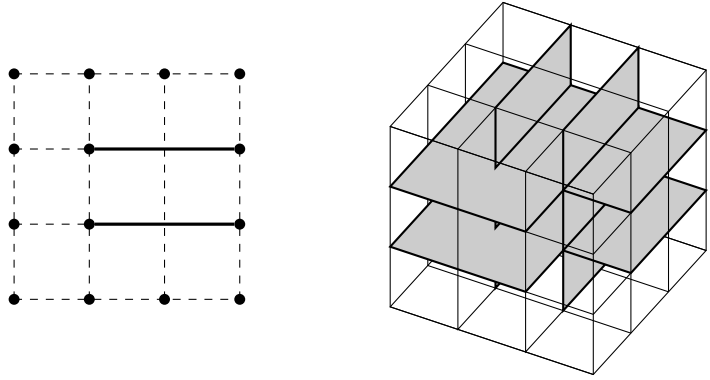

  We construct $T=T_{n,k}$ by induction on $n$ and $k$.  For $k=0$ we can set
  $T_{n,0}=\emptyset$.  Similarly, for $k=n$ we can take $T_{n,n}$ to be $K$ less
  the interior of any one cell---for concreteness, let that be the cell that
  includes the origin.

  Now we construct $T_{n,k}$ for $n>k>0$ by induction on $n$.  Write $K=K^\prime
  \times I$; then for every $0<i<r$, let $T|_{K^\prime \times \{i/r\}}=T_{n-1,k}
  \times \{i/r\}$, and for every $0 \leq i<r$, let
  $$T|_{K^\prime \times (i/r,i+1/r)}=T_{n-1,k-1} \times (i/r,i+1/r).$$
  Finally, we throw in $K^{k-1} \cap K^\prime \times \{0,1\}$.  It remains to
  show that the resulting complex $T=T_{n,k}$ satisfies the lemma.

  It is clear that $T$ contains $K^{(k-1)}$ and that condition (1) holds.
  Moreover, $T$ deformation retracts first to
  $T|_{\partial K \cup K^\prime \times \{i/r: i=1,2,\ldots,r-1\}}$ using a retraction of
  $T_{n-1,k-1}$, and thence to $(\partial K)^{(k-1)}$ via a retraction of each
  layer individually.  This demonstrates (2).  It remains to show that (3) and
  (4) hold.

  In order to do this more easily, we present an alternate rule for determining
  whether a $k$-cell $c$ is contained in $T$.  Showing that it is indeed
  equivalent to the previous definition is tedious but straightforward.  Let $c$
  be a $k$-cell of $K \setminus \partial K$ and let $\mathcal{I}(c) \subset
  \{1,\ldots,n\}$ be the set of directions in which it has positive width.
  Write $\pi_i$ for the projection of $K$ onto its $i$th interval factor, and
  $\ell(c)$ for the greatest integer such that $\{1,\ldots,\ell\} \subset
  \mathcal{I}(c)$.  Then $c$ is in $T$ if and only if $\pi_ic \neq [0,1/r]$ for
  some $i \leq \ell$.  In particular, if $\ell=0$, then $c \notin T$.

  Now let $c$ be a $k$-cell of $K \setminus T$.  If $c \in \partial K$, then it
  already fits the bill, so suppose it is in $K \setminus (T \cup \partial K)$.
  We will argue that $c$ is bounded rel $T \cup \partial K$ by a box $B$ with
  positive width in directions $\mathcal{I}(c) \cup \{\ell(c)+1\}$.
  Specifically, the projections of $B$ onto each interval factor of $K$ are as
  follows:
  \begin{align*}
    \pi_iB &= I & \text{if }1 \leq i \leq \ell(c); \\
    \pi_iB &= [x,1]\text{ where }\pi_ic=\{x\} & \text{if }i = \ell(c)+1; \\
    \pi_iB &= \pi_ic & \text{otherwise.}
  \end{align*}
  By the criterion above, $\partial B \setminus \partial K$ contains only one
  cell which is not in $T$, namely $c$.  Thus $\partial B \cap \partial K$ is
  the chain desired for (3).

  More generally, given a $k$-dimensional box in $K$, one can take the union of
  the $B$'s constructed for each cell in the box.  This gives a solution for
  (4).
\end{proof}
\begin{proof}[Proof of Lemma \ref{lem:sub} for cubulations.]
  Let $X_L$ be the complex obtained by dividing $X$ into grids at scale $1/L$.
  We begin by choosing a $k$-spanning tree $T$ for $X$, then use it to build a
  $k$-spanning tree $T_L$ for $X_L$.  We include all cells of $X_L$ contained in
  $T$; for every cell of $X$ not contained in $T$, we include a complex as in
  Lemma \ref{lem:cube_tree}.  The resulting subcomplex includes $X_L^{(k-1)}$
  and, by induction on $n-k$, deformation retracts to $T$.  Therefore it is a
  $k$-spanning tree for $X_L$.

  Now let $\mathcal{B}$ be an optimal basis for $H_k(X,T;\mathbb{Q}) \cong
  H_k(X_L,T_L;\mathbb{Q})$.  By the argument above, any $k$-cell of $X_L$ is
  homologous rel $T_L$ to a sum of at most $2^{n-k}n!/k!$ cells in the $k$-faces,
  where $n$ is the dimension of $X$.  In turn, any $k$-cell of
  $X_L \setminus T_L$ which is contained in a $k$-face represents the same
  homology class modulo $T_L$ as that face does modulo $T$, and therefore can be
  represented as a sum of at most $G(T)$ elements of $\mathcal{B}$.  Therefore,
  $G(T_L) \leq G(T) \cdot 2^{n-k}n!/k!$.
\end{proof}
We now have the tools we need to prove Lemma \ref{lem0'}.  We will do this by
way of two auxiliary lemmas.  The first states that any cochain with
coefficients in $\mathbb{F}/\mathbb{Z}$ which can be lifted to $\mathbb{F}$ can
be lifted to a cochain which is not too big.
\begin{lem}[bounded lifting] \label{lem2}
  Let $X$ be a finite simplicial complex and $z \in
  Z^k(X;\mathbb{F}/\mathbb{Z})$ a cocycle, and let $T$ be a $k$-spanning tree of
  $X$.  Then if $z$ lifts to a cocycle $\tilde z \in Z^k(X;\mathbb{F})$, we can
  find such a lift $\tilde z$ with $\lVert \tilde z \rVert_\infty \leq k+1+G(T)$.
\end{lem}
\begin{proof}
  Fix $U$, a $(k-1)$-wrapping tree of $X$.  Then for every $(k-1)$-simplex $p$
  of $X \setminus U$, there is a unique $k$-chain $F(p)$ supported in $T$ which
  fills $p$ mod $U$.  Moreover,
  $$\mathcal{F}=\{F(p): p\text{ is a $(k-1)$-simplex of }X \setminus U\}$$
  is a basis for $C_k(T)$: they are linearly independent since their boundaries
  are linearly independent in $C_{k-1}(X)$, and any $k$-simplex $q$ in $T$ can be
  expressed as an integral linear combination $\sum_{p \in \partial q} F(p)$.  We
  can therefore extend $F$ by linearity to an isomorphism
  $F:C_{k-1}(X;\mathbb{F}) \to C_k(T;\mathbb{F})$.

  Now let $\mathcal{B}$ be an optimal basis for $H_k(X,T;\mathbb{F})$ which
  demonstrates that $T$ is $G(T)$-gnarled.  For every $b \in \mathcal{B}$,
  choose a $\hat b \in C_k(X,T;\mathbb{F})$ representing it and let
  $$\tilde{\mathcal{B}}=\{\hat b-F(\partial \hat b): b \in \mathcal{B}\}.$$
  These are cycles and form a basis for $H_k(X;\mathbb{F})$.

  Now for any cocycle $w \in C^k(T;\mathbb{F})$ and any $k$-simplex $q$ of $X$,
  we can write
  $$\langle w, q\rangle=\left\langle w, \sum_{p \in \partial q} F(p)\right\rangle+
  \left\langle w, q-\sum_{p \in \partial q} F(p)\right\rangle.$$
  The chain $q-\sum_{p \in \partial q} F(p)$ is a cycle, and hence homologous to the
  sum of at most $G(T)$ elements of $\tilde{\mathcal{B}}$ (with signs.)  Thus
  $w$ is determined by its values on $\mathcal{F} \cup \tilde{\mathcal{B}}$.
  Conversely, any function $\mathcal{F} \cup \tilde{\mathcal{B}} \to \mathbb{F}$
  extends to a $k$-cocycle on $X$: the values on $\mathcal{F}$ determine its
  values on simplices of $T$, while the values on $\tilde{\mathcal{B}}$
  determine its values on cycles.  Since there are no cycles in $T$, these are
  independent.

  Now let $\tilde z_0$ be any lift of $z$ to a cocycle in $C^k(X;\mathbb{F})$.
  If we change $\tilde z_0$ by changing the values on
  $\mathcal{F} \cup \tilde{\mathcal{B}}$ by integers, we get a new cocycle; in
  particular, we can do this to get a new $\tilde z$ such that its values on
  $\mathcal{F} \cup \tilde{\mathcal{B}}$ are in $[0,1)$.  Now, for every
  $k$-simplex $q$, $\langle \tilde z, q \rangle=
  \sum \pm\langle \tilde z, c \rangle$ where the sum is over $k+1+G(T)$ elements
  $c \in \mathcal{F} \cup \tilde{\mathcal{B}}$.  Therefore, $\tilde z$ is still
  a lift of $z$ and has $\lVert \tilde z \rVert_\infty \leq k+1+G(T)$.
\end{proof}
Now we show that if a chain has a filling with $\mathbb{Z}$ coefficients, we can
find such a filling near any filling with $\mathbb{F}$ coefficients.
\begin{lem} \label{Zapprox}
  Let $X$ be a finite simplicial complex equipped with the standard metric, and
  let $X_L$ be the cubical or edgewise $L$-regular subdivision of $X$, and let
  $k \geq 0$.  Then there is a constant $C(X,k)$ such that for any $\alpha \in
  C^k(X_L;\mathbb{F})$ such that $\delta\alpha$ takes integer values and is a
  coboundary over $\mathbb{Z}$, there is an $\tilde\alpha \in
  C^k(X_L;\mathbb{Z})$ such that $\delta\alpha=\delta\tilde\alpha$ and
  $\lVert\tilde{\alpha}\rVert_\infty \leq \lVert\alpha\rVert_\infty+C$.
\end{lem}
\begin{proof}
  By Lemma \ref{lem:sub}, $X_L$ admits a spanning tree whose gnarledness is
  bounded by a constant $C_0(X,k)$.  Then by Lemma \ref{lem2}, the cocycle
  $\alpha \mod \mathbb{Z} \in C^k(X_L;\mathbb{F}/\mathbb{Z})$ has a lift
  $\Delta\alpha \in C^k(X_L;\mathbb{F})$ with $\lVert\Delta\alpha\rVert_\infty
  \leq k+1+C_0$.  Then we can set $\tilde\alpha=\alpha-\Delta\alpha$ and
  $C=k+1+C_0$.
\end{proof}
\begin{proof}[Proof of Lemma \ref{lem0'}]
  If $\omega=0$, we can take $\alpha=0$, so suppose $\omega \neq 0$.

  By Lemma \ref{lem0F}, we can find an $\alpha \in C^{k-1}(X_L;\mathbb{Q})$ such
  that $d\alpha=\omega$ and $\lVert\alpha\rVert_\infty \leq
  KL\lVert\omega\rVert_\infty$.  Then by Lemma \ref{Zapprox} we can find an
  $\tilde\alpha \in C^{k-1}(X_L;\mathbb{Z})$ such that $d\tilde\alpha=\omega$ and
  $$\lVert\tilde\alpha\rVert_\infty \leq KL\lVert\omega\rVert_\infty+k+1+C_0 \leq
  (KL+k+1+C_0)\lVert\omega\rVert_\infty.$$
  This gives us an estimate for the isoperimetric constant $C_{\mathrm{IP}}(X,k)$.
\end{proof}

\section{Building linear homotopies}

In this section, we prove Theorem \ref{intro:main}.  The proof is based on two
lemmas: one to take care of obstructions posed by finite homotopy groups, and
the other for infinite obstructions.

We start with a fairly general result for finite homotopy groups.  It shows that
if a map $X \to Z$ can be retracted to a subspace $Y \subset Z$ with finite
relative homotopy groups, then one can force this retraction to be geometrically
bounded.  The special case in which $Y$ is a point is proven in \cite{FWPNAS} as
Theorem 1.
\begin{lem} \label{lem:Qlift}
  Let $Y \subset Z$ be a pair of finite simplicial complexes such that
  $\pi_k(Z,Y)$ is finite for $k \leq n+1$.  Then there is a constant $C(n,Y,Z)$
  with the following property.  Let $X$ be an $n$-dimensional simplicial complex
  and $f:X \to Z$ a simplicial map which is homotopic to a map $g:X \to Y$.
  Then there is a short homotopy of $f$ to a map $g^\prime$ which is homotopic to
  $g$ in $Y$, that is, a homotopy which is $C$-Lipschitz under the standard
  metric on the product cell structure on $X \times [0,1]$.
\end{lem}
Note that the constant $C$ does not depend on $X$ and in particular on the
choice of a subdivision of $X$.  Thus if we consider Lipschitz and not just
simplicial maps from $X$ to $Y$, the width of the homotopy remains constant,
rather than linear in the Lipschitz constant as is the case with some of our
later results.

We will actually use the following relative version of this result: if $f:(X,A)
\to (Z,Y)$ homotopes into $Y$ rel $A$, then there is a corresponding short
homotopy rel $A$.  The proof below works just as well for this variant; one
merely has to check that at every stage $f|_A$ remains invariant.
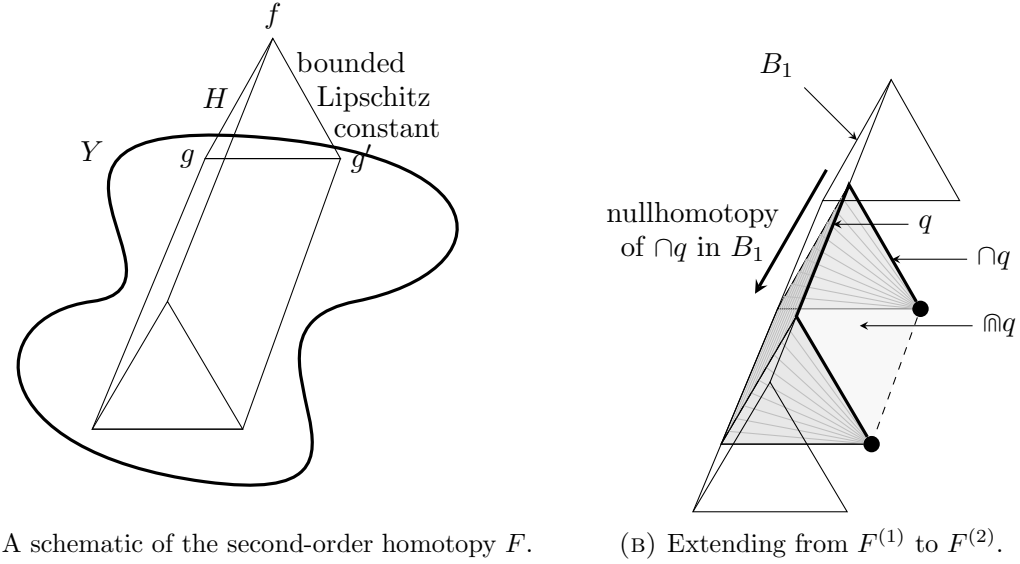
\begin{figure} 
  \centering
  \begin{subfigure}[b]{0.5\textwidth}
    \centering
  \begin{tikzpicture}
      \node(Y) at (0.5,4) {$Y$};
      \draw[very thick] (2,-0.4) .. controls (-1,0) and (-1,1.8) ..
      (0.5,2) .. controls (2,2.2) and (-1,4.4) ..
      (2.5,4.2) .. controls (6,4) and (6,2.4) ..
      (4,2) .. controls (2,1.6) and (5,-0.8) .. cycle;
      \coordinate (a1) at (0.5,0.3);
      \coordinate (a2) at (1.5,2);
      \coordinate (a3) at (2.5,0.3);
      \coordinate (b1) at (2,3.9);
      \coordinate (b2) at (2.9,5.5);
      \coordinate (b3) at (3.8,3.9);
      \draw (a1) -- (a2) -- (a3) -- cycle;
      \draw (b1) -- (b2) -- (b3) -- cycle;
      \draw (a1) -- (b1) -- (b3) -- (a3) -- cycle;
      \node [above] at (b2) {$f$};
      \node [left] at (b1) {$g$};
      \node [left] at (barycentric cs:b1=0.5,b2=0.5) {$H$};
      \node [right] at (barycentric cs:b3=0.2,b2=0.8) {bounded};
      \node [right] at (barycentric cs:b3=0.5,b2=0.5) {Lipschitz};
      \node [right] at (barycentric cs:b3=0.75,b2=0.25) {constant};
      \node [right] at (b3) {$g^\prime$};
      \draw (a2) -- (b2);
  \end{tikzpicture}
  \caption{A schematic of the second-order homotopy $F$.}
  \end{subfigure}
  \begin{subfigure}[b]{0.4\textwidth}
    \centering
    \begin{tikzpicture}
      \tikzstyle{chainfill}=[fill=gray!25, fill opacity=0.7];
      \tikzstyle{lightchainfill}=[fill=gray!10, fill opacity=0.5];
      \tikzstyle{dott}=[circle,fill=black,scale=0.6];
      \coordinate (a1) at (0.5,0.3);
      \coordinate (a0) at (1.5,2);
      \coordinate (a2) at (2.5,0.3);
      \coordinate (b1) at (1.25,2.1);
      \coordinate (b0) at (2.2,3.75);
      \coordinate (b2) at (3.15,2.1);
      \coordinate (w0) at (barycentric cs:b0=-0.5,a0=1.5);
      \coordinate (w1) at (barycentric cs:b1=-0.5,a1=1.5);
      \coordinate (w2) at (barycentric cs:b2=-0.5,a2=1.5);
      \coordinate (z0) at (barycentric cs:b0=1.8,a0=-0.8);
      \coordinate (z1) at (barycentric cs:b1=1.8,a1=-0.8);
      \coordinate (z2) at (barycentric cs:b2=1.8,a2=-0.8);
      \filldraw[chainfill] (a1) -- (a0) -- (b0) -- (b1) -- cycle;
      \filldraw[chainfill] (b1) -- (b2) -- (b0) -- cycle;
      \filldraw[chainfill] (a1) -- (a2) -- (a0) -- cycle;
      \fill[lightchainfill] (a0) -- (a2) -- (b2) -- (b0) -- cycle;
      \node[dott] (A2) at (a2) {};
      \node[dott] (B2) at (b2) {};
      \foreach \t in {0.1,0.2,...,0.9} {
        \draw[gray!50] (A2) -- (barycentric cs:a0=\t,a1=1-\t) --
        (barycentric cs:b0=\t,b1=1-\t) -- (B2);
      }
      \draw[very thick] (A2) -- (a0) -- (b0) -- (B2);
      \draw (a0) -- (a1) -- (A2);
      \draw (w0) -- (w1) -- (w2) -- cycle (z0) -- (z1) -- (z2) -- cycle;
      \draw (w0) -- (z0) (w1) -- (z1);
      \draw[dashed] (a2) -- (b2);
      \draw[-stealth,very thick] (2.2-0.3,3.75+0.2) --
      node[anchor=east,align=center] {nullhomotopy\\of $\cap q$ in $B_1$}
      (1.25-0.3,2.1+0.2);
      \draw[stealth-] (barycentric cs:b0=0.7,a0=0.3) -- +(1,0)
      node[anchor=west] {$q$};
      \draw[stealth-] (barycentric cs:b0=0.4,b2=0.6) -- +(1,0)
      node[anchor=west] {$\cap q$};
      \draw[stealth-] (barycentric cs:a0=0.45,b2=0.45,a2=0.1) -- +(1.5,0)
      node[anchor=west] {$\Cap q$};
      \draw[stealth-] (barycentric cs:z0=0.5,z1=0.5) -- +(-0.7,0.7)
      node[anchor=south east] {$B_1$};
    \end{tikzpicture}
      \caption{
          Extending from $F^{(1)}$ to $F^{(2)}$.
      }
  \end{subfigure}
  \caption{Illustrations for the proof of Lemma \ref{lem:Qlift}.}
\end{figure}
\begin{proof}
  Let $H:X \times [0,1] \to Z$ be a homotopy with $H_0=f$ and $H_1=g$; we have
  no control over this homotopy, only over $f$.  Our strategy will be to push
  both $f$ and the homotopy into $Y$ via a second-order homotopy.  Let
  $\Delta^2$ be the 2-simplex with edges $e_0$, $e_1$, and $e_2$ opposite
  vertices $0$, $1$, and $2$.  At the end of the construction, we will obtain a
  map $F:X \times \Delta^2 \to Z$ such that $F|_{e_2}=H$, $F|_{e_0}$ lands in $Y$,
  and $F|_{e_1}:X \times [0,1] \to Z$ is the short homotopy we are looking for.

  We will construct this map one skeleton of $X$ at a time.  At each step we
  ensure that the restrictions $F|_{q \times e_1}$ for simplices $q$ of $X$ are
  chosen from a finite set of Lipschitz maps depending only on $Y$ and $Z$.  In
  this way we get a universal bound on the Lipschitz constant.  We start by
  setting $B_0=X \times e_2$ and
  $$F^{(0)}=H:B_0 \to Z.$$
  In general, for $k \geq 0$, let
  $$B_k=(X^{(k-1)} \times \Delta^2) \cup (X \times e_2)$$
  and $A_k=B_k \cap X \times e_0$.  Then suppose by induction we have a map
  $$F^{(k)}:(B_k,A_k) \to (Z,Y)$$
  such that the restrictions $F|_{q \times e_1}$ for $(k-1)$-simplices $q$ of $X$
  are contained in a finite set $\mathcal{F}_k(Z,Y)$.  We would now like to
  extend this (over cells of the form $q \times \Delta^2$, for every $k$-simplex
  $q$ of $X$) to a map $F^{(k+1)}:(B_{k+1},A_{k+1}) \to (Z,Y)$.

  To avoid doing an extra ad hoc step we will use the convention $S^{-1}=
  \emptyset$.  Let $k \geq 0$.  Given a $k$-simplex $q$, let
  $$\cap q=(q \times \{0\} \cup \partial q \times e_1, \partial q \times \{2\})
  \subset (B_k,A_k).$$
  We think of this as a map $(D^k,S^{k-1}) \to (B_k,A_k)$.  It can be homotoped
  into $A_k$ rel boundary via a nullhomotopy which is constant on the
  $X$-coordinate and sends $e_1$ to $e_0$, keeping the vertex $\{2\}$ constant.
  Therefore the map $F^{(k)} \circ \cap q:(D^k,S^{k-1}) \to (Z,Y)$ homotopes rel
  boundary into $Y$.  Moreover, the set of homotopy classes of maps homotoping
  $F^{(k)} \circ \cap q$ into $Y$ (more precisely, of maps
  $$(q \times e_1,q \times \{2\}) \to (Z,Y)$$
  which restrict to $F^{(k)} \circ \cap q$ on $\cap q$) is in (non-canonical)
  bijection with $\pi_{k+1}(Z,Y)$.  One such bijection $u_q$ is obtained by
  sending a map $\ph:(q \times e_1,q \times \{2\}) \to (Z,Y)$ to the map
  $$u_q(\ph):\Cap q:=(q \times (e_1 \cup e_2) \cup \partial q \times \Delta^2,
  q \times \{1,2\} \cup \partial q \times e_0) \to (Z,Y)$$
  which restricts to $\ph$ on $(q \times e_1,q \times \{2\})$ and to $F^{(k)}$
  everywhere else.

  Now, by our inductive assumption, the number of different values the map
  $F^{(k)} \circ \cap q$ may take is bounded above by
  $$\lvert\mathcal{F}_k(Z,Y)\rvert^{k+1} \cdot
  \#\{(k-1)\text{-simplices of }Z\}.$$
  Let $\mathcal{F}_{k+1}(Z,Y)$ contain one Lipschitz map
  $$(\Delta^k \times e_1,\Delta^k \times \{2\}) \to (Z,Y)$$
  for each possible value of $F^{(k)} \circ \cap q$ and each homotopy class of
  nullhomotopy; thus there are at most
  $$\lvert\mathcal{F}_k(Z,Y)\rvert^{k+1} \cdot \#\{(k-1)\text{-simplices of }Z\}
  \cdot \lvert\pi_{k+1}(Z,Y)\rvert$$
  such maps.  We then set $F^{(k+1)}|_{q \times e_2}$ to be the element $\ph$ of
  $\mathcal{F}_{k+1}(Z,Y)$ for which $u_q([\ph])=0$.  With this choice, the map
  can be extended in some way to $q \times \Delta^2$.  Since this part of the
  map does not need to be controlled, we can do this in an arbitrary way.

  At the end of the induction, we have our map $F$: the Lipschitz constant of
  $F|_{X \times e_1}$ is at most $\max\{\Lip \ph: \ph \in \mathcal{F}_{n+1}(Z,Y)\}$.
\end{proof}
Now we prove Theorem \ref{intro:main} in the case where the target space is an
Eilenberg--MacLane space.  This will also be incorporated into the proof of the
general case.
\begin{thm} \label{thm:ooo}
  Let $X$ be a finite $n$-dimensional simplicial complex and $Y$ a finite
  simplicial complex with an $(n+1)$-connected map $Y \to K(\mathbb{Z},m)$, for
  some $m \geq 2$.  Then there are constants $C_1(n,Y)$ and $C_{\mathrm{IP}}(X,m)$
  such that any two homotopic $L$-Lipschitz maps $f,g:X \to Y$ are
  $C_1C_{\mathrm{IP}}(L+1)$-Lipschitz homotopic through $C_1(L+1)$-Lipschitz maps.
\end{thm}
This theorem is the main geometric input into the proof of Theorem
\ref{intro:main}, and is by itself enough to prove certain important cases.  For
example, it shows directly that any $L$-Lipschitz map $f:S^3 \to
\mathbb{C}\mathbf{P}^2$ is $CL$-nullhomotopic, as is any nullhomotopic
$L$-Lipschitz map $X \to S^n$ for any $n$-dimensional $X$.  The general proof
strategy is that described in \S\ref{S:S2}.
\begin{proof}
  $Y$ is homotopy equivalent to the CW complex obtained from it by contracting
  an $m$-spanning tree.  In order to create maps which we can homotope
  combinatorially, we simplicially approximate $f$ and $g$ on an $L$-regular
  subdivision of $X$, then compose with this contraction.  After the homotopy is
  constructed, we can compose with the homotopy equivalence going back to get to
  the original $Y$.  This increases constants multiplicatively and adds short
  homotopies to the ends; both of these can be absorbed into $C_1$.

  For the rest of the proof we assume that $Y$ is the contracted complex and
  that $f$ and $g$ are compositions of simplicial maps with the contraction.

  We construct the homotopy by induction on skeleta of $X \times I$.  In
  particular $f(X^{(m-1)})=\{*\}$.  Let $C_{\mathrm{IP}}=C_{\mathrm{IP}}(X,m)$ be the
  isoperimetric constant from Lemma \ref{lem0'}, and let $Z$ be the polyhedral
  complex given by the product cell structure on $X \times I$, where $I$ is
  split into $C_{\mathrm{IP}}L$ subintervals
  $[i/C_{\mathrm{IP}}L,(i+1)/C_{\mathrm{IP}}L]$.  We define
  $F_{m-1}: X \times \{0,1\} \cup Z^{(m-1)} \to Y$ by setting
  $F_{m-1}|_{X \times \{0\}}=f|_{X^{(m)}}$ and $F_{m-1}|_{X \times \{1\}}=g|_{X^{(m)}}$ and
  sending the rest to $*$.

  Now define a simplicial cocycle $\omega \in C^m(X;\pi_m(Y))$ by setting
  $$\langle\omega,q\rangle=\left[f|_{(q,\partial q)}\right]
  -\left[g|_{(q,\partial q)}\right] \in \pi_m(Y)$$
  for $m$-simplices $q$ of $X$.  Since $Y$ has a finite number of cells, there
  is a finite number of possible values of $\omega$ on simplices.  In
  particular, $\lVert\omega\rVert_\infty \leq C$ for some $C=C(Y)$.

  By assumption, since $f \simeq g$, $\omega$ is a coboundary.  By Lemma
  \ref{lem0'}, $\omega=d\alpha$ for some cochain $\alpha \in C^{m-1}(X;\pi_m(Y))$
  with $\lVert\alpha_i\rVert_\infty \leq C_{\mathrm{IP}}CL$.  We will use
  $\alpha$ to construct a cochain $\beta \in C^m(Z;\pi_m(Y))$ which we will use
  to extend $F_{m-1}$ to $Z^{(m)}$.

  Define $\beta$ as
  $$\begin{aligned}
    \left\langle\beta, p \times \left[\frac{i}{C_2L},\frac{i+1}{C_2L}\right]
    \right\rangle &
    =\left\lfloor\frac{i+1}{C_2L}\langle\alpha, p\rangle\right\rfloor-
    \left\lfloor\frac{i}{C_2L}\langle\alpha, p\rangle\right\rfloor &
    \begin{array}{r}
      \text{for $(m-1)$-simplices $p$ of }X,\\
      0 \leq i<C_2L;
    \end{array} \\
    \langle\beta, q \times \{i/C_2L\}\rangle &
    =-\langle\omega, q\rangle+\sum_{p \in \partial q}
    \left\lfloor\frac{i}{C_2L}\langle\alpha, p\rangle\right\rfloor &
    \begin{array}{r}
      \text{for $m$-simplices $q$ of }X,\\
      0 \leq i<C_{\mathrm{IP}}L.
    \end{array}
  \end{aligned}$$
  Clearly, $\beta$ is a cocycle.  Moreover, since
  $\left\lvert\sum_{p \in \partial q} \langle\alpha, p\rangle \right\rvert =
  \lvert\langle\omega,q\rangle\rvert \leq C$, one can see that
  $$\lVert\beta\rVert_\infty \leq C+m+1.$$
  In particular the bound depends only on $Y$.

  For each possible value of $\beta$ on cells, choose representatives
  $(\Delta_m,\partial\Delta_m) \to (Y,*)$ and $(\Delta_{m-1} \times I,
  \partial(\Delta_{m-1} \times I)) \to (Y,*)$ and extend $F_{m-1}$ to each
  $m$-cell of $Z$ using the appropriate representative to get
  $F|_{X \times \{0,1\} \cup Z^{(m)}}$.  By construction, for each $(m+1)$-cell $c$
  of $Z$, $F|_{\partial c}$ is nullhomotopic.

  Now suppose we have constructed $F|_{Z^{(k)}}$ for some $m \leq k \leq n$.  By
  induction, there is a finite number, depending only on $k$ and $Y$, of
  possible restrictions $F|_{\partial c}$, where $c$ is a $(k+1)$-cell of $Z$.
  Moreover, if $k \geq m+1$, $F|_{\partial c}$ is nullhomotopic since $\pi_k(Y)
  \cong 0$.  Thus for each possible restriction $F|_{\partial c}$, we can choose
  an extension to $c$.  Extending $F$ to $X \times \{0,1\} \cup Z^{(k+1)}$ in
  this way gives us a finite set, depending on $k+1$ and $Y$, of possible
  restrictions to $(k+2)$-cells.

  At the conclusion of the induction, we obtain a map $F$ which is the desired
  nullhomotopy.
\end{proof}
In general, the constant $C_1$ increases by a multiplicative factor in each
dimension, depending on the topology of $Y$.  It is worth attempting to analyze
$C_1$ and $C_{\mathrm{IP}}$ in simple cases, for example for maps $S^2 \to S^2$.
Here, simplicial approximation multiplies the Lipschitz constant by slightly
more than $2\sqrt{3}$.  The induction has one step, and if $\omega$ satisfies
$\lVert\omega\rVert_\infty=1$, then $\beta$ satisfies $\lVert\beta\rVert_\infty
\leq 4$.  With a bit of care in plumbing as we connect preimages of $S^2
\setminus *$ on the surface of our 3-cells, we can build the nullhomotopy by
increasing the Lipschitz constant by a factor of 3.  This gives a total
multiplicative factor of $C_1=6\sqrt{3}+\epsi \approx 10.4$ when $L$ is large.
The isoperimetric constant $C_{\mathrm{IP}}$ depends on the exact geometric model
for the preimage sphere; in the case of the tetrahedron, it is 1.

Putting together Lemma \ref{lem:Qlift} and Theorem \ref{thm:ooo}, we can now
prove Theorem \ref*{intro:main}.  We recall this result below:
\begin{thm*}
  Let $X$ be an $n$-dimensional finite complex.  If $Y$ is a finite simply
  connected complex which is rationally equivalent through dimension $n$ to a
  product of Eilenberg--MacLane spaces, then there are constants $C_1(n,Y)$ and
  $C_2(X)$ such that homotopic $L$-Lipschitz maps from $X$ to $Y$ are
  $C_1C_2(L+1)$-Lipschitz homotopic through $C_1(L+C_2)$-Lipschitz maps.
\end{thm*}
A corollary for highly connected $Y$ follows from the rational Hurewicz theorem.
\begin{cor} \label{cor:highly_connected}
  Let $Y$ be a rationally $(k-1)$-connected finite complex and $X$ an
  $n$-dimensional finite complex.  Then if $n \leq 2k-2$, then there are
  constants $C_1(n,Y)$ and $C_2(X)$ such that homotopic $L$-Lipschitz maps from
  $X$ to $Y$ are $C_1C_2(L+1)$-Lipschitz homotopic through
  $C_1(L+C_2)$-Lipschitz maps.
\end{cor}
Before giving the proofs of the corollary and the theorem, we recall some facts
about maps to Eilenberg--MacLane spaces which derive from properties of the
obstruction-theoretic isomorphism
$$[(X,A),(K(G,n),*)] \cong H^n(X,A;G)$$
induced by cell-wise degrees on cellular maps.  See for example Chapter 8 of
\cite{Spa} for details.  Let $X$ be any CW complex, $n \geq 2$, and $G$ an
abelian group, and consider a CW model of $K(G,n)$ whose $(n-1)$-skeleton is a
point $*$.  Then:
\begin{itemize}
\item $K(G,n)$ is an H-space.  That is, the element in $\Hom(G^2,G)$ sending
  $(a,b) \mapsto ab$ induces a multiplication map $\mult:K(G,n) \times K(G,n)
  \to K(G,n)$.  This has identity $*$, i.e.~it sends
  $$(K(G,n) \times *) \cup (* \times K(G,n)) \mapsto *,$$
  and is associative and commutative up to homotopy.  It can also be assumed
  cellular.
\item Let $f:X \to K(G,n)$ be a map.  Then the group
  $$\pi_1(\Map(X,K(G,n)),f) \cong [X \times [0,1],K(G,n)]_f$$
  of self-homotopies of $f$ is naturally isomorphic to $H^{n-1}(X;G)$.
\item Denote the map that sends $X$ to $* \in K(G,n)$ also by $*$.  Then
  $$\pi_1(\Map(X,K(G,n)),*) \cong [SX,K(G,n)]$$
  acts freely and transitively on $\pi_1(\Map(X,K(G,n)),f)$ via the
  multiplication map; the above isomorphism takes this to the action of
  $H^{n-1}(X;G)$ on itself via multiplication.
\end{itemize}
\begin{proof}[Proof of Corollary.]
  The rational Hurewicz theorem (see e.g.~\cite{KK}) states that if $X$ is a
  simply connected space such that $\pi_i(X) \otimes \mathbb{Q}=0$ for
  $i \leq k-1$, then the Hurewicz map
  $$\pi_i(X) \otimes \mathbb{Q} \to H_i(X;\mathbb{Q})$$
  induces an isomorphism for $i \leq 2k-2$.  Therefore, for $i \leq 2k-2$,
  $$[X,K(\pi_i(X) \otimes \mathbb{Q},i)]
  \cong H^i(X;\pi_i(X) \otimes \mathbb{Q})
  \cong \Hom(\pi_i(X) \otimes \mathbb{Q},\pi_i(X) \otimes \mathbb{Q}).$$
  In particular, we can find a map $\ph_i:X\to K(\pi_i(X) \otimes \mathbb{Q},i)$
  which induces the identity on $\pi_i$.  Then the map
  $$(\ph_2,\ph_3,\ldots,\ph_{2k-2}):X \to
  \prod_{i=1}^{2k-2} K(\pi_i(X) \otimes \mathbb{Q},i)$$
  is rationally $(2k-1)$-connected.  This allows us to apply Theorem
  \ref{intro:main}.
\end{proof}
\begin{proof}[Proof of Theorem \ref{intro:main}]
  Suppose that $Y$ is rationally homotopy equivalent through dimension $n$ to
  $\prod_{i=1}^r K(\mathbb{Z},n_i)$.  This gives us a map $Q:Y \to
  \prod_{i=1}^r K(\mathbb{Q},n_i)$ inducing an isomorphism on
  $H^*({-};\mathbb{Q})$.  For each $i$, let $\alpha_i \in H^{n_i}(Y;\mathbb{Z})$
  be in the preimage of the copy of $\mathbb{Q}$ corresponding to
  $H^{n_i}(K(\mathbb{Q},n_i))$; this induces a map $\ph_i:Y \to
  K(\mathbb{Z},n_i)$.  Then
  $$\ph=(\ph_1,\ldots,\ph_r):Y \to Z=\prod_{i=1}^r Z_i$$
  is again a rational homology isomorphism and so by the rational Hurewicz
  theorem, $(Z,Y)$ is a pair with $\pi_k(Z,Y)$ finite for $k \leq n+1$.

  Let $f,g:X \to Y$ be homotopic $L$-Lipschitz maps, and let $C_{2,k}=
  C_{\mathrm{IP}}(X,k)$.  Then by Theorem \ref{thm:ooo}, for each $i$, there is a
  $C_{1,i}(Y)$ such that $\ph_i \circ f$ and $\ph_i \circ g$ are
  $C_{1,i}C_{2,n_i}L$-Lipschitz nullhomotopic through $C_{1,i}L$-Lipschitz maps via
  homotopies $F_i:X \times [0,1] \to Z_i$.  Then
  $$F:=(F_1,\ldots,F_r):X \times [0,1] \to Z$$
  is a $\sum_{i=1}^r C_{1,i}C_2L$-Lipschitz homotopy.  Suppose first that we can
  homotope $F$ to an uncontrolled homotopy of $f$ and $g$ in $Y$.  Then by the
  relative version of Lemma \ref{lem:Qlift} applied to the pair
  $(X \times [0,1],X \times \{0,1\})$, there is a $C_1(n,Y,Z)$ such that $f$ and
  $g$ are $C_1C_2L$-Lipschitz homotopic in $Y$ through $C_1L$-Lipschitz maps.

  Note that such a homotopy may not exist a priori; we will need to modify $F$
  so that it does.  For this we use an algebraic construction.  We know that
  there is some homotopy $G:X \times [0,1] \to Y$ between $f$ and $g$.  So we
  can concatenate the homotopies $F$ and $\ph \circ G$ to give a map
  $H:X \times S^1 \to Z$:
  $$H(x,t)=\left\{\begin{array}{l l}
  F(x,2t) & 0 \leq t \leq 1/2 \\
  \ph \circ G(x,2(1-t)) & 1/2 \leq t \leq 1,
  \end{array}\right.$$
  where we identify $S^1=\mathbb{R}/\mathbb{Z}$.  Then $H$ represents an element
  of $\pi_1(\Map(X,Z),\ph \circ f)$.  Since each factor $Z_i$ is a
  high-dimensional skeleton of an H-space, there is a multiplication map
  $\mult:Z^{(M)} \times Z^{(M)} \to Z$ for some large enough $M$.  This induces a
  free transitive action of $[SX,Z]$ on each $\pi_1(\Map(X,Z),\ph \circ f)$.

  We now analyze the cokernel of the group homomorphism
  $$\pi_1(\Map(X,Y),f) \to \pi_1(\Map(X,Z),\ph \circ f).$$
  Consider the relative Postnikov tower
  \begin{center}
    \begin{tikzpicture}
      \node (x) at (0,0) {$Y$};
      \node (x0) at (3,0) {$P_1$};
      \node (x0prime) at (4.1,0) {$=P_0=Z$};
      \node (x1) at (3,1) {$P_2$};
      \node (vd) at (3,2) {$\vdots$};
      \node (xn) at (3,3) {$P_n$};
      \draw[->] (x) -- (x0) node[near end,anchor=south,inner sep=1pt]
           {$\ph_0=\ph$};
      \draw[->] (x) -- (x1) node[midway,anchor=south,inner sep=2pt] {$\ph_n$};
      \draw[->] (x) -- (xn) node[midway,anchor=south east,inner sep=1pt]
           {$\ph_n$};
      \draw[->] (xn) -- (vd) node[midway,anchor=west] {$p_n$};
      \draw[->] (vd) -- (x1) node[midway,anchor=west] {$p_3$};
      \draw[->] (x1) -- (x0) node[midway,anchor=west] {$p_2$};
    \end{tikzpicture}
  \end{center}
  of the inclusion $\ph:Y \hookrightarrow Z$.  Here, $P_k$ is a space such that
  $\pi_i(P_k,Y)=0$ for $i \leq k$ and $\pi_i(Z,P_k)=0$ for $i>k$.  The map $p_k$
  therefore only has one nonzero relative homotopy group, $\pi_k(Z,Y)$.  In this
  setting there is an obstruction theory long exact sequence
  (\hspace{1sp}\cite[\S2.5]{Baues}; cf.~also \cite[Prop.~14.3]{GM} and
  \cite[Lemma~2.7]{SulGen}) of groups
  \begin{multline*}
    \cdots \to H^{k-1}(X;\pi_k(Z,Y)) \to \pi_1(\Map(X,P_k),\ph_k \circ f) \to \\
    \to \pi_1(\Map(X,P_{k-1}),\ph_{k-1} \circ f) \to H^k(X;\pi_k(Z,Y)) \to \cdots.
  \end{multline*}
  In particular, an element of $\lvert\pi_k(Z,Y)\rvert
  \pi_1(\Map(X,P_{k-1}),\ph_{k-1} \circ f)$ is the image of some loop of maps to
  $P_k$ based at $\ph_k \circ f$.  Hence, independently of $\ph \circ f$,
  $$R\pi_1(\Map(X,Z),\ph \circ f)\text{, where }R:=\prod_{k=2}^n
  \lvert\pi_k(Z,Y)\rvert$$
  always lifts to $\pi_1(\Map(X,Y),f)$.  Let $\mathcal{H}$ be the (finite!)
  collection of linear combinations with coefficients between $0$ and $R-1$ of
  some finite generating set for $[SX,Z]$.  Then for any $f:X \to Y$, the finite
  set
  $$\{\mult(a,\id_{\pi_1(\Map(X,Z),\ph \circ f)}): a \in \mathcal{H}\}$$
  surjects onto the cokernel we are interested in.

  We can then choose $a \in \mathcal{H}$ so that $\mult([a],[H])$ can be
  homotoped into $Y$.  Now define a map $\tilde H:X \times S^1 \to Z$ by
  $$\tilde H(x,t)=\left\{\begin{array}{l l}
  \mult(F(x,2t),a(x,2t)) & 0 \leq t \leq 1/2 \\
  \ph \circ G(x,2(1-t)) & 1/2 \leq t \leq 1.
  \end{array}\right.$$
  Then $\tilde H$ is in the same homotopy class as $\mult([a],[H])$.  Therefore
  the map $\tilde F:X \times [0,1] \to Z$ given by $\tilde F(x,t)=
  \tilde H(x,t/2)$ is a homotopy between $f$ and $g$ which homotopes into $Y$,
  and whose Lipschitz constant is bounded by
  $$\Lip(\mult) \cdot \left(\Lip F+\max_{a \in \mathcal{H}} \Lip a\right).$$ 
  This is linear in $\max\{\Lip f,\Lip g\}$, and except for $\Lip a$, the
  coefficients depend only on $n$, $Y$ and $Z$, so $\tilde F$ can be plugged
  into the argument above.
\end{proof}
\begin{rmk}
  Note that in this proof, the dependence of $\max_{a \in \mathcal{H}} \Lip a$ on
  $X$ lies only in the choice of generating set for $[SX,Z]$.  In certain
  special cases, this constant can be independent of $X$.  For example, suppose
  that we know that $X$ is an $n$-sphere (or even just an $n$-dimensional PL
  homology sphere.)  Then $[SX,Z]=\pi_{n+1}(Z)$ is generated by maps whose degree
  on simplices is at most 1---regardless of the geometry of $X$.  This means
  that for such homology spheres $X$, $L$-Lipschitz maps $f,g:X \to Y$ can be
  homotoped through maps of Lipschitz constant $C(Y)L$, though the width of the
  homotopy required may depend on the geometry.  This may have applications like
  finding skinny metric tubes between ``comparable'' metrics on the sphere.  In
  contrast, results of Nabutovsky and Weinberger imply that without this
  comparability condition, such tubes may have to be extremely (uncomputably)
  thick.
\end{rmk}

\section{A counterexample}

One may ask whether the linear bound of Theorem \ref{intro:main} holds for any
simply-connected target space, not just products of Eilenberg--MacLane spaces.
The answer is emphatically no.  Here we give, for each $n \geq 4$, a space $Y$
and a sequence of nullhomotopic maps $S^n \to Y$ such that volume of any
Lipschitz nullhomotopy grows faster than the $(n+1)$st power of the Lipschitz
constant of the maps.  This forces the Lipschitz constant of the
nullhomotopy to grow superlinearly.

To make this precise: by the volume of a map $F:S^n \times [0,1] \to Y$ we mean
$$\vol F=\int_{S^n \times [0,1]} \lvert\Jac F(x)\rvert d\vol$$
(recall that by Rademacher's theorem the derivative of a Lipschitz map is
defined almost everywhere.)  By this definition,
$$\vol F \leq \vol(S^n \times [0,1])\sup_{x \in S^n \times [0,1]}
\lvert\Jac F(x)\rvert \leq \vol(S^n \times [0,1])(\Lip F)^{n+1}.$$

To construct the space $Y$, we take $S^2 \vee S^2$ and attach $(n+1)$-cells via
attaching maps which form a basis for $\pi_n(S^2 \vee S^2) \otimes \mathbb{Q}$.
Note that by rational homotopy theory, $\pi_{*+1}(S^2 \vee S^2) \otimes
\mathbb{Q}$ is a free graded Lie algebra on two generators of degree 1 whose Lie
bracket is the Whitehead product (see Exercise 44 of \cite{GM} or Example 1 of
\cite[\S24(f)]{FHT}.)  In particular, if $f$ and $g$ are the identity maps on
the two copies of $S^2$, the iterated Whitehead product
$$h_1=[f,[f,\ldots[f,g]\ldots]]:S^n \to S^2 \vee S^2,$$
with $f$ repeated $n-2$ times, represents a nonzero element of
$\pi_n(S^2 \vee S^2)$.  Moreover, the map
$$h_L=[L^2f,[L^2f,\ldots[L^2f,L^2g]\ldots]]:S^n \to S^2 \vee S^2$$
is an $O(L)$-Lipschitz representative of $L^{2n-2}[h_1]$.  Thus in $Y$, we can
define a nullhomotopy $H$ of $h_L$ by first homotoping it inside
$S^2 \vee S^2$ to $h_1 \circ \ph_{2n-2}$ for some map $\ph_{2n-2}:S^n \to S^n$ of
degree $L^{2n-2}$, and then nullhomotoping each copy of $h_1$ via a standard
nullhomotopy.

Since $h_1$ is not nullhomotopic in $S^2 \vee S^2$, this standard nullhomotopy
must have degree $C \neq 0$ on at least one of the $(n+1)$-cells, giving a
closed $(n+1)$-form $\omega$ on $Y$ such that $\int_{S^n \times I} \omega^*H=
L^{2n-2}C$.  Now, suppose $H^\prime$ is some other nullhomotopy of $h_L$.  Then
gluing $H$ and $H^\prime$ along the copies of $S^n \times \{0\}$ gives a map
$p:S^{n+1} \to Y$.  Note that if any map $(D^{n+1},S^n) \to (Y,S^2 \vee S^2)$ had
nonzero degree on cells, then the map $S^n \to S^2 \vee S^2$ on the boundary
would be homotopically nontrivial.  This shows that $p$ must have total degree
zero on cells, in other words, that $\int_{S^n \times I} \omega^*H^\prime=L^{2n-2}C$.
Thus the volume of a nullhomotopy of $h_L$ grows at least as $L^{2n-2}$.

In the sequel to this paper, we show that for $n=4$, this estimate is sharp, in
the sense that we can always produce a nullhomotopy whose Lipschitz constant is
quadratic in the time coordinate and linear in the others.

\section{Quantitative cobordism theory}

The goal of the rest of the article is to prove Theorem \ref*{filling}, which we
recall below.
\begin{thm*}
  If $M$ is an oriented closed smooth null-cobordant manifold which admits a
  metric of bounded local geometry and volume $V$, then it has a null-cobordism
  which admits a metric of bounded local geometry and volume
  $$\leq c_1(n) V^{c_2(n)}.$$
  Moreover, $c_2(n)$ can be chosen to be $O(\exp(n))$.
\end{thm*}

As described in the introduction, we will prove this theorem by executing the
following steps.  We begin by choosing a metric $g$ on $M$ such that $(M,g)$ has
bounded local geometry, and such that the volume $V$ of $(M,g)$ is bounded by
twice the complexity of $M$.  We then proceed as follows:
\begin{enumerate}
\item We embed $M$ into $\mathbb{R}^{n+k}$ for an appropriately large $k$
  (depending on $n$) so that the embedding has bounded curvature, bounded
  volume, and has a large tubular neighborhood.  We will use this map to embed
  the manifold into the standard round sphere $\mathbb{S}^{n+k}$ while
  maintaining bounds on its geometry.
  \vspace{2mm}
\item We show that the Pontryagin--Thom map from this sphere to the Thom space
  of the universal bundle of oriented $k$-planes in $\mathbb{R}^{n+k}$ (relative
  to the embedded manifold and its tubular neighborhood) has Lipschitz constant
  bounded as a function of $n$ and the volume of $M$.
  \vspace{2mm}
\item We analyze the rational homotopy type of the Thom space and determine
  that, up to dimension $n+k+1$, it is rationally equivalent to a product of
  Eilenberg--MacLane spaces.  Since $M$ is null-cobordant, this map is
  nullhomotopic, and so as a result, we can apply Theorem \ref*{intro:main} to
  conclude that there is a nullhomotopy which has Lipschitz constant bounded as
  a function of $n$ and the volume of $M$.  This translates to a map from the
  ball with boundary $\mathbb{S}^{n+k}$ to the Thom space with the same bound on
  the Lipschitz constant.
  \vspace{2mm}
\item The proof is completed by simplicially approximating this map from the
  ball, then using PL transversality theory to obtain an $(n+1)$-dimensional
  manifold, embedded in this ball, which fills $M$ and satisfies the conclusions
  of the theorem.
\end{enumerate}

Throughout this section, we will use the following notation.  We write
$x \lesssim y$ to mean that there is a constant $c(n)>0$, depending only on $n$,
such that $x \leq c(n) y$; $x \gtrsim y$ is defined analogously.  Similarly, we
write $x \lesssim A^{\lesssim y}$ to imply that there are constants $c_1(n)>0$ and
$c_2(n)>0$, again depending only on $n$, such that $x \leq c_1(n)A^{c_2(n) y}$.
We define the same expression with $\gtrsim$ analogously.  Throughout this
section we will also use $V$ to denote the volume of $M$.  Lastly, we will write
$Gr(n+k,n)$ to denote the Grassmannian of oriented $n$-dimensional planes in
$\mathbb{R}^{n+k}$ and $Th(n+k,n)$ to denote the Thom space of the universal
bundle over this Grassmannian.  $Gr(n+k,n)$ is given the standard metric,
which induces a metric on $Th(n+k,n)$.  Furthermore, we denote by $p^*$ the
basepoint of the Thom space $Th(n+k,n)$.

We begin by explicitly defining what ``bounded local geometry'' means in Theorem
\ref*{filling}.

\begin{defn} \label{defn:quantitative_bounded_geometry}
  Suppose that $(M, g)$ is a closed Riemannian manifold of dimension $n$.
  Following \cite{CheeGr}, we say that $M$ has \emph{bounded local geometry}
  $\geo(M) \leq \beta$ if it has the following properties:
  \begin{enumerate}[label=(B\arabic*)]
  \item	$M$ has injectivity radius at least $1/\beta$.
  \item	All elements of the curvature tensor are bounded below by $-\beta^2$ and
    above by $\beta^2$.
  \end{enumerate}
  The manifold $(M,g)$ satisfies $\widetilde{\geo}(M) \leq \beta$ if in
  addition it satisfies the following condition:
  \begin{enumerate}[label=(B\arabic*),resume]
  \item	The $k$th covariant derivatives of the curvature tensor are bounded by
    constants $C(n,k)\beta^{k+2}$.  (The $C(n,k)$ are defined once and for all,
    but we will not specify them.)
  \end{enumerate}
\end{defn}
Conditions (B1)--(B3) taken together agree with the standard definition used by
Riemannian geometers, except that we require explicit quantitative bounds.  A
theorem of Cheeger and Gromov \cite[Thm.~2.5]{CheeGr} states that for any given
$\epsi>0$ a metric $g$ on $M$ with $\geo(M,g) \leq 1$ can be $\epsi$-perturbed
to $g_\epsi$ with $\geo(M,g_\epsi) \lesssim 1$ which satisfies (B3).  In
particular, $\vol(M,g_\epsi) \leq (1+\epsi)^n\vol(M,g)$.  By rescaling, we get a
metric $\hat g$ with $\geo(M,\hat g) \leq 1$ and
$$vol(M,\hat g) \lesssim (1+\epsi)^n\vol(M,g).$$
Therefore, for the rest of the proof we can assume that (B3) holds, with a
constant multiplicative penalty on the volume of our manifold.

Finally, if $M$ has boundary, we say, following \cite{Schick}, that it satisfies
$\geo(M) \leq \beta$ if (B1) holds at distance at least $\beta$ from the
boundary, (B2) holds everywhere, and in addition the neighborhood of
$\partial M$ of width 1 is isometric to a collar $\partial M \times [0,\beta]$.
In particular, this implies that $\geo(\partial M) \leq \beta$.

\subsection{Embedding $M$ into $\mathbb{R}^{n+k}$}

To begin constructing the embedding described in the first step, we first choose
a suitable atlas of $M$.  A similar set of properties defines uniformly regular
Riemannian manifols, a notion due to H.~Amann (see, for example, page 4 of
\cite{DMS}.)  However, we require our quantitative bounds on the geometry of the
maps to be much more uniform, depending only on the dimension; we also require
that the charts can be partitioned into a uniform number of subsets consisting
of pairwise disjoint charts.
\begin{lem} \label{lem:chart_existence}
  Suppose that $M$ is a compact orientable $n$-dimensional manifold with
  $\widetilde{\geo}(M) \leq 1$.  There exists a finite atlas $\mathfrak{U}$ with
  the following properties, expressed in terms of constants $\mu \leq 3/25$,
  $c$, and $q$ depending only on $n$, as well as a natural number
  $10 \leq m \leq \kappa\exp(\kappa n)$ for some constant $\kappa>0$.
  \begin{enumerate}
  \item	Every map in $\mathfrak{U}$ is the exponential map from the Euclidean
    $n$-ball of radius $\mu$ to $M$ which agrees with the orientation of $M$.
    Since the injectivity radius of $M$ is at least $1$ and $\mu < 1$, this is
    well-defined.  We write $\mathfrak{U}=\{\phi_i:B_{\mu} \rightarrow M_{\mu}\}$.
    Here, $M_{\mu}$ is a geodesic ball of $M$ of radius $\mu$, and $B_{\mu}$ is
    the Euclidean ball of radius $\mu$ in $\mathbb{R}^n$.
    \vspace{2mm}
  \item	$\mathfrak{U}$ can be written as the disjoint union of sets
    $\mathfrak{U}_1,\ldots,\mathfrak{U}_m$ of charts such that any pair of
    charts from the same $\mathfrak{U}_j$ have disjoint image.
    \vspace{2mm}
  \item	When we restrict all the maps in $\mathfrak{U}$ to $B_{\mu/4}$, they
    still cover $M$.
    \vspace{2mm}
  \item	The pullback of the metric with respect to every $\phi \in \mathfrak{U}$
    is comparable to the Euclidean metric, that is,
    $$\frac{1}{q}(\rho \cdot \rho) \leq \phi^* g(x)(\rho, \rho)
    \leq q (\rho \cdot \rho)$$
    for every $\rho \in \mathbb{R}^n$, for every $x \in T_xM$, and where
    $\phi^* g(x)$ is the  pullback of $g$ at $x$. 
    \vspace{2mm}
  \item	The first and second derivatives of all transition maps are bounded by
    $c$.
  \end{enumerate}
\end{lem}
\begin{proof}
  As mentioned above, this list of properties is closely related to one used in
  the definition of a uniformly regular Riemannian manifold.  Every compact
  manifold is uniformly regular, and it is known that a (potentially
  non-compact) orientable manifold $M$ with $\geo(M) \leq \beta$ for some
  $\beta$ is uniformly regular; this is shown in \cite{Amann2015}.  This
  guarantees an atlas with properties similar, though not identical, to the
  above.  We use a similar set of arguments to those compiled by Amann.

  To begin, we cover $M$ by balls of radius $\frac{\mu}{12} \leq \frac{1}{100}$.
  Since $M$ is compact, we require only finitely many balls to cover $M$.
  Furthermore, by the Vitali covering lemma, we can choose a finite subset
  $B_1, \ldots, B_k$ of these balls such that $3B_1, \ldots, 3B_k$ also cover
  $M$, and such that $B_1, \ldots, B_k$ are disjoint.  We also have that the
  balls $12B_1, \dots, 12B_k$ cover $M$, and that these balls have radius $\mu$.

  Fix a ball $12B_i$ for some $i$.  We would like to count how many other balls
  in $12B_1, \ldots, 12B_k$ intersect $12B_i$.  Call these balls $12B_{j_1},
  \ldots, 12 B_{j_m}$.  Then $B_i$ and $B_{j_1}, \dots, B_{j_m}$ all lie inside
  $50B_i$, and all are disjoint.  Since $M$ has bounded local geometry, the
  volume of $50B_i$ is bounded above in terms of $n$, and the volumes of $B_i$
  and $B_{j_1}, \dots, B_{j_p}$ are bounded below in terms of $n$.  This yields an
  exponential bound on $m$ in terms of $n$.  As a result, the balls $12B_1,
  \dots, 12B_k$ can be partitioned into $m$ sets of pairwise disjoint balls.
  We define these sets as $\mathcal{B}_1, \dots, \mathcal{B}_m$.  This proof is
  analogous to a standard proof of the Besicovitch covering lemma in
  $\mathbb{R}^n$.

  For every $j$ with $1 \leq j \leq m$, $\mathfrak{U}_j$ is defined as follows.
  For every ball $B \in \mathcal{B}_i$, the exponential map goes from the
  Euclidean ball of radius $\mu$ to $B$; furthermore, it can be chosen so that
  it agrees with the orientation of $M$.  These are exactly the charts that
  comprise $\mathfrak{U}_j$.  The first three properties that we desire are now
  satisfied.

  Property (4) is part of Lemma 1 of \cite{HKW}.  Indeed, all $k$th derivatives
  of the metric tensor are also bounded by a constant depending only on $n$ and
  $k$ \cite{Eich} \cite{Schick}.  This allows us to also bound the derivatives
  of the pullback of the Euclidean metric along transition functions between the
  charts.  Property (5) follows immediately from this.
\end{proof}

We will also need the following simple observation.
\begin{lem}
  \label{lem:function}
  There is a $C^\infty$ function $\zeta$ from $[0,\mu]$ to $[0,1]$ such that:
  \begin{enumerate}
  \item	$\zeta$ is monotonically increasing with $\zeta(0)=0$ and
    $\zeta(\mu)=1$.
  \item	$\zeta(t) = t/\mu$ for all $t \in [0,\mu/2]$.
  \item	$\zeta^{(k)}(\mu) = 0$ for all $k \in \mathbb{Z}_{> 0}$.
  \item	For every $k \in \mathbb{Z}_{\geq 0}$, there is some
    $c(k) \in \mathbb{R}$ such that
    $$\lvert\zeta^{(k)}(t)\rvert \leq c(k)$$
    for all $t \in [0,\mu]$.
  \end{enumerate}
\end{lem}

We will now embed $M$ into $\mathbb{R}^{n+k}$ so that we have control over its
geometry.  In particular, we will prove the following proposition.
\begin{prop}
  \label{prop:embedding}
  Suppose that $(M^n,g)$ is a compact orientable $n$-dimensional Riemannian
  manifold with volume $V$ and bounded local geometry.  Then there is some
  $k$ which depends on $n$ such that $M$ is diffeomorphic to a submanifold
  $M' \subset \mathbb{R}^{n+k}$, $k \leq \kappa(n+1)\exp(\kappa n)$, with the
  following properties:
  \begin{enumerate}
    \vspace{2mm}
  \item	$M'$ lies in a ball of radius $\lesssim 1$.
    \vspace{2mm}
  \item	Let $F:M' \rightarrow Gr(n+k,n)$ be the smooth map sending $x \in M'$ to
    $T_x M' \subset \mathbb{R}^{n+k}$, the oriented tangent space of $M'$ at $x$.
    Then $F$ has Lipschitz constant $\lesssim 1$.
    \vspace{2mm}
  \item	$\widetilde{M}$ has a normal tubular neighborhood of size $\gtrsim 1/V$.
\end{enumerate}
\end{prop}
\begin{proof}
  We will use the chart $\mathfrak{U}$ constructed in Lemma
  \ref*{lem:chart_existence} to define an embedding of $(M,g)$ into
  $\mathbb{R}^N$, with $N>2n+3$ depending only on $n$.  By property (2),
  $\mathfrak{U}$ can be written as a disjoint union
  $\bigsqcup_{j=1}^m \mathfrak{U}_j$ of sets of charts with disjoint images.  The
  number of elements in each $\mathfrak{U}_j$ is $\lesssim V$ since these
  disjoint images have volume $\gtrsim 1$.  Let $R=\max_{1 \leq j \leq m}
  \#\mathfrak{U}_j$.  We define $n$-dimensional spheres $\mathbb{S}_1, \dots,
  \mathbb{S}_R$ in $\mathbb{R}^{n+1}$ by the following properties:
  \begin{enumerate}
  \item	$\mathbb{S}_i$ has radius $1 + i/R$;
  \item	every $\mathbb{S}_i$ passes through the origin;
  \item	the center of every $\mathbb{S}_i$ lies on the ray from the origin in
    the direction $(1,0,\dots,0)$.
  \end{enumerate}
  The radii of the spheres are between $1$ and $2$, and the difference between
  any two of the radii is $\gtrsim 1/V$.  An example of such a sequence of
  spheres is shown in Figure \ref*{fig:spheres}.  We will refer to the antipode
  of the origin on each sphere as its ``north pole''.
  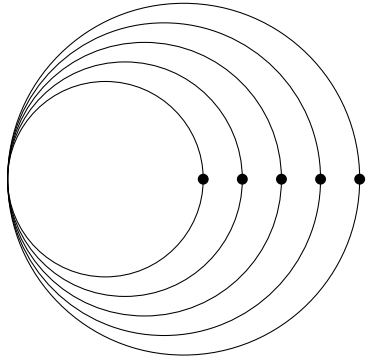
\begin{figure} 
    \centering
    \begin{tikzpicture}[scale=1.3]
      \tikzstyle{dott}=[circle,fill=black,scale=0.4];
      \foreach \x in {0,0.2,...,0.8} {
        \draw (\x,0) circle (1+\x);
        \node[dott] (l\x) at (2*\x+1,0) {};
      }
    \end{tikzpicture}
    \caption{A sequence of $1$-spheres in $\mathbb{R}^2$, with north poles
      spaced at distance $1/V$.}
    \label{fig:spheres}
  \end{figure}

  Define $N=m(n + 1)$ and $k=N-n$.  Fix a point $x \in M$.  Our embedding
  $E: M \rightarrow M' \subset (\mathbb{R}^{n+1})^m$ will map $x$ to
  $(\vec y_1,\ldots,\vec y_m)$, where each $\vec y_j \in \mathbb{R}^{n+1}$, as
  follows.  For every $j$ with $1 \leq j \leq m$, if $x$ is not in the image of
  any chart of $\mathfrak{U}_j$, then we set $\vec y_j=\vec 0$.  If not, then
  $x$ is in the image of exactly one chart $\phi_i: B_{\mu} \rightarrow M$ in
  $\mathfrak{U}_j$, $1 \leq i \leq R$.  In this case, we set $\vec y_j$ to be
  the point on $\mathbb{S}_i$ given by composing $\phi_i^{-1}(x)$ with a map
  $\kappa_i:B_{\mu}\rightarrow \mathbb{S}_i$ which is defined as follows: take
  the origin to the north pole of $\mathbb{S}_i$, and then map the geodesic
  sphere of radius $r$ in $B_{\mu}$ homothetically to the geodesic sphere around
  the north pole in $\mathbb{S}_i$ of radius $\zeta(r) D_i$.  Here $\zeta$ is
  defined as in Lemma \ref{lem:function} and $D_i$ is the intrinsic diameter of
  $\mathbb{S}_i$.

  Define a map $\widehat{\phi_i}:M \to \mathbb{R}^{n+1}$ by
  $$\widehat{\phi_i}=\begin{cases}
  \kappa_i \circ \phi_i^{-1}(x) & x \in \phi_i(B_\mu) \\
  \vec 0 & \text{otherwise.}
  \end{cases}$$
  Since $\zeta$ is smooth and all its derivatives go to 0 at $\mu$, this is a
  smooth map whose derivative has rank $n$ on $B_\mu$.  If the original charts
  in $\mathfrak{U}_j$ are $(\phi^j_1, \dots, \phi^j_{q_j})$, then we can write
  $$E(x)=\left(\sum_{i=1}^{q_1} \widehat{\phi^1_i}(x),\ldots,
  \sum_{i=1}^{q_m} \widehat{\phi^m_i}(x)\right).$$
  Since $\mathfrak{U}$ is an atlas, for any $x$, some $\widehat{\phi^j_i}(x)$ is
  nonzero.  On the other hand, at most one of $\widehat{\phi^j_1}(x),\ldots,
  \widehat{\phi^j_{q_j}}(x)$ is nonzero.  This shows that $E$ is an immersion.
  Moreover, if $E(x_1)=E(x_2)$, then for some chart $x_1$ and $x_2$ are in the
  image of that chart, and in fact $x_1=x_2$.  This shows that $E$ is injective.
  Since every $S_i$ is contained in a ball of radius $2$ around the origin,
  every point in $M$ is mapped to a point in $\mathbb{R}^{n+k}$ of norm
  $\lesssim 1$.

  We have a natural set $\mathfrak{U}^\prime$ of oriented charts for the embedded
  manifold $M'$ given by $E \circ \phi^j_i$ for each $\phi^j_i \in
  \mathfrak{U}$.  Since the first and second derivatives of all of the
  transition maps are bounded $\lesssim 1$, since $\zeta$ has bounded
  derivatives, and since the radii of the balls are all bounded below by $1$ and
  above by $2$, the first and second derivatives of all charts are $\lesssim 1$.
  Moreover, since every point of $M$ is contained in $\phi^j_i(B_{\mu/4})$ for
  some $i$ and $j$, and $\frac{d\zeta}{dt}=1$ for $t \leq \mu/4$, the first
  derivative of each chart is $\gtrsim 1$.

  Combined with the property that the pullback of the metric of $M$ using each
  chart $\phi^j_i$ is comparable to the Euclidean metric, this shows that the
  map from $M$ to $M^\prime$ with its intrinsic Riemannian metric is bilipschitz
  with constant $\lesssim 1$.

  Let us now consider the map $F$ as defined in the statement of the
  proposition.  Fix a point $x^\prime \in M^\prime$, and choose one of the above
  charts $\phi^\prime$ which covers $x^\prime$, and define $x \in B_{\mu}$ to be
  the unique point with $\phi^\prime(x) = x^\prime$.  Choose unit vectors
  $v_1, \dots, v_n$ in $\mathbb{R}^n$ such that
  $$\frac{D_{v_1} \phi^\prime(x)}{\lvert D_{v_1} \phi^\prime(x) \rvert}, \dots,
  \frac{D_{v_n} \phi^\prime(x)}{\lvert D_{v_n} \phi^\prime(x) \rvert}$$
  is an orthonormal set of vectors that spans the tangent plane of $M^\prime$ at
  $x^\prime$.  For any unit vector $w \in \mathbb{R}^n$, consider
  $$\left\lvert D_w \frac{D_{v_1} \phi^\prime(x)}
  {\lvert D_{v_1} \phi^\prime(x) \rvert} \right\lvert, \dots,
  \left\lvert D_w \frac{D_{v_n} \phi^\prime(x)}
              {\lvert D_{v_n} \phi^\prime(x) \rvert} \right\lvert.$$
  Since all first and second derivatives of $\phi'$ are bounded above by
  $\lesssim 1$, and since the first derivatives of $\phi'$ are bounded from
  below by $\gtrsim 1$, all of these values are bounded by $\lesssim 1$.  Since
  the original vectors are orthonormal, for $\epsilon$ sufficiently small the
  distance in $Gr(N,n)$ between the tangent plane at $\phi'(x)$and the tangent
  plane at $\phi'(x + \epsilon w)$ is $\lesssim \epsilon$.  Since $\phi^\prime$
  is $\lesssim 1$-bilipschitz, this completes the proof that $F$ is
  $\lesssim 1$-Lipschitz.
	
  Lastly, we want to show that $M^\prime$ has a normal tubular neighborhood of
  width $\gtrsim 1/V$.  Suppose that $x^\prime$ and $y^\prime$ are two points on
  $M^\prime$ and $v_{x'}$ and $v_{y'}$ are two normal vectors such that
  $x^\prime+v_{x'}=y^\prime+v_{y'}$.  We would like to show that
  $\max(|v_{x'}|, |v_{y'}|) \rvert \gtrsim 1/V$.

  Let $\theta$ be the angle between $v_{x'}$ and $v_{y'}$.  Consider a
  minimal-length geodesic $\gamma$, parametrized by arclength, between $x^\prime$
  and $y^\prime$; $v_{x^\prime}$ and $v_{y^\prime}$ lie in the orthogonal
  $(N-1)$-planes to this geodesic at $x^\prime$ and $y^\prime$, respectively.  The
  above arguments imply that the tautological embedding $M^\prime \to
  \mathbb{R}^N$ has second derivatives $\lesssim 1$.  Therefore, the second
  derivative of $\gamma$ is $\lesssim 1$.
  \begin{prop*}
    Let $\ell=\len(\gamma)$.  Then $\ell \gtrsim \theta$.
  \end{prop*}
  \begin{proof}
    Let $V$ be the plane spanned by $v_{x'}$ and $v_{y'}$, and let $\pi_V$ and
    $\pi_{V^\perp}$ be orthogonal projections to $V$ and $V^\perp$.  Then:
    \begin{itemize}
    \item the average over $[0,\ell]$ of $\pi_{V^\perp}\frac{d\gamma}{dt}$ is $0$;
    \item $\frac{d\gamma}{dt}(0) \cdot v_{x'}=0$ and
      $\frac{d\gamma}{dt}(\ell) \cdot v_{y'}=0$.
    \end{itemize}
    The bounds on the second derivative then imply that for every $t$,
    $$\pi_{V^\perp}\frac{d\gamma}{dt} \lesssim \ell\text{ and }
    \pi_V\frac{d\gamma}{dt} \lesssim \frac{\ell}{\sin(\theta/2)} \lesssim
    \frac{\ell}{\theta}.$$
    Therefore
    $$\ell=\int_0^\ell \sqrt{\left\lvert\pi_V\frac{d\gamma}{dt}\right\rvert^2
      +\left\lvert\pi_{V^\perp}\frac{d\gamma}{dt}\right\rvert^2}dt \lesssim
    \ell^2\sqrt{\theta^{-2}+1},$$
    and therefore $\ell \gtrsim \frac{\theta}{\sqrt{1+\theta^2}} \gtrsim
    \theta$.
  \end{proof}

  Now let $\phi$ be a chart in some $\mathfrak{U}_j$ such that $x^\prime \in
  E \circ \phi(B_{\mu/4})$.  Suppose first that $y^\prime \in E \circ
  \phi(B_{\mu/2})$.  Then the properties of any $\kappa_i$ imply that
  $\lvert x^\prime-y^\prime \rvert \gtrsim \text{length}(\gamma)$; in particular,
  $\lvert x^\prime-y^\prime \rvert \gtrsim \theta$ and so
  $\max(|v_{x'}|, |v_{y'}|) \gtrsim 1$.

  On the other hand, suppose that $y^\prime$ is not in $E \circ \phi(B_{\mu/2})$.
  Suppose first that it is in $E \circ \phi(B_\mu)$ but not
  $E \circ \phi(B_{\mu/2})$.  Here again the properties of any $\kappa_i$ imply
  that $x^\prime-y^\prime \gtrsim 1$.  The same is true if $y^\prime$ is not in
  the image of any $\phi^\prime \in \mathfrak{U}_j$.  Finally, if $y^\prime$ is in
  $\phi^\prime \in \mathfrak{U}_j$ for some $\phi^\prime \neq \phi$, then the
  properties of the $\kappa_i$ imply that $x^\prime-y^\prime \gtrsim 1/V$.  In all
  these cases it must be the case that
  $$\max(|v_{x'}|, |v_{y'}|) \rvert \geq \frac{x^\prime-y^\prime}{2} \gtrsim 1/V.$$

  This completes the proof that $M^\prime$ has a large tubular neighborhood.
\end{proof}

Finally we prove a lemma which allows us to embed $M'$ into a round sphere.
\begin{lem} \label{lem:sphere_embedding}
  Suppose that $M'$ is an embedded submanifold of $\mathbb{R}^{n+k}$ which
  satisfies all of the conclusions of Proposition \ref*{prop:embedding}.  Then
  there is an embedding $\widetilde{E}: M' \rightarrow \widetilde{M} \subset
  \mathbb{S}^{n+k}$ into the round unit sphere such that
  \begin{enumerate}
  \item	$\widetilde{M}$ has a tubular neighborhood of width $\gtrsim 1/V$.
    Additionally, $\widetilde{E}$ can be extended to a $\lesssim 1$-Lipschitz
    diffeomorphism from this tubular neighborhood to a neighborhood of width
    $\gtrsim 1/V$ of $M'$.
  \item	The map $\widetilde{F}: \widetilde{M} \rightarrow Gr(n+k,n)$ given by
    $F \circ \widetilde{E}^{-1}$ has Lipschitz constant $\lesssim 1$.  Here, $F$
    is the map from $M'$ to $Gr(n+k,n)$ from Proposition \ref*{prop:embedding}.
  \end{enumerate}
\end{lem}
\begin{proof}
  $M'$ is contained in a ball of radius $\lesssim 1$, and without loss of
  generality we may assume that this ball is centered at the origin.  If we
  restrict the stereographic projection to $M'$, we obtain an embedded manifold
  of $\mathbb{S}^{n+k}$ which satisfies all of the above properties.
\end{proof}

\subsection{Proof of Theorem \ref*{filling}}

To complete the proof of Theorem \ref*{filling}, we use the embedding of
$M$ in $\mathbb{S}^{n+k}$ produced by combining Proposition \ref*{prop:embedding}
with Lemma \ref*{lem:sphere_embedding}.  We begin by describing the
Pontryagin--Thom map, and by computing its Lipschitz constant.

We map $\mathbb{S}^{n+k}$ into $Y = Th(n+k,n)$, the Thom space of the universal
bundle of oriented $n$-dimensional planes in $\mathbb{R}^{n+k}$, via a map
$G:\mathbb{S}^{n+k} \to Y$ defined as follows.  Let $z \in \mathbb{S}^{n+k}$.  If
$z$ is outside of the tubular neighborhood of $\widetilde{M}$ of width
$c_1(n)/V$ (here the constant depending on $n$ is the same as that in Lemma
\ref*{lem:sphere_embedding}), then it is mapped to $p^*$ (the basepoint of
$Th(n+k,n)$).  If not, then applying $\widetilde{E}^{-1}$ to $z$ produces a
point in the tubular neighborhood of $M'$ of width $c_2(n)/V$ (this constant
depending on $n$ is the same as that in Proposition \ref*{prop:embedding}).
Hence, $\widetilde{E}^{-1}(z) = x + y$ where $x \in M'$ and $y$ is a point in
the oriented normal plane $\mathcal{N}$ of $M'$ at $x$, and $y$ has length
$< c_2(n)/V$.  Both $x$ and $y$ are unique.  We then take
$G(z)=\left(\mathcal{N},\frac{V}{c_2(n)}y\right) \in Th(n + k, n)$.

Since the map $\widetilde{F}$ from Lemma \ref*{lem:sphere_embedding} is
Lipschitz with Lipschitz constant $\lesssim 1$, the map from
$x \in \widetilde{M}$ to the oriented normal plane of $M'$ at
$\widetilde{E}^{-1}(x)$ is also Lipschitz with Lipschitz constant $\lesssim 1$.
If we assume that $c_2(n)/V$ is at most half the critical radius of the tubular
neighborhood, then the projection $z \mapsto x$ has Lipschitz constant $\leq 2$.
Furthermore, the tubular neighborhood of $M'$ is dilated by a factor of
$\lesssim V$ when it is mapped to $Th(n+k,n)$, and the map $\widetilde{E}^{-1}$
has Lipschitz constant $\lesssim 1$ on the tubular neighborhood of width
$c_1(n) V$ of $\widetilde{M}$.  Hence, the Lipschitz constant of $G$ is
$\lesssim V$.

By \cite[Theorem~of~Thom,~page~215]{milnor1974characteristic}, the map $G$ is
nullhomotopic, since $M$ (and so $M'$ and $\widetilde{M}$ with the orientation
induced by the charts $\phi'$ as in the proof of Proposition
\ref*{prop:embedding} and the stereographic projection from Lemma
\ref*{lem:sphere_embedding}) is null-cobordant.  $Th(n+k,n)$ is
$(k-1)$-connected by \cite[Lemma~18.1]{milnor1974characteristic}.  We can
assume, perhaps by adding extra ``empty'' dimensions, that $k>n+3$ and
therefore $2(k-1)>n+k+1$.

By Corollary \ref{cor:highly_connected}, since $Th(n+k,k)$ is a metric CW
complex, there is a nullhomotopy of $G$ with Lipschitz constant
$\lesssim C_{\mathbb{S}^{n+k}, Th(n+k,n)} V$.  This constant depends only on $n$, and
so there is a nullhomotopy $H$ of $G$ of Lipschitz constant $\lesssim V$.  This
extends to a map from a ball $B$ of radius $1$ in $\mathbb{R}^{n+k+1}$ to
$Th(n+k,n)$ with Lipschitz constant $\lesssim V$.

We now observe that we can consider both $B$ and $Y = Th(n+k,n)$ as finite
simplicial complexes in the following sense.  Since the result follows from
standard arguments, we omit the proof.
\begin{lem} \label{lem:simplicial_approximation}
  There is a finite simplicial complex $\widetilde{Y}$ and a scale $L_1(n)$ such
  that if we give each simplex the metric of the standard simplex of side length
  $L_1(n)$, then there is a $2$-bilipschitz function $f_Y$ from $Y$ to
  $\widetilde{Y}$.  Furthermore, the image of the zero-section of $Y$ under this
  map is a subcomplex (and a simplicial submanifold) of $\widetilde{Y}$.

  Similarly, there is a finite simplicial complex $\widetilde{B}$ and a scale
  $L_2(n)$ such that if every simplex is given the metric of the standard
  simplex of side length $L_2(n)$, then there is a $2$-bilipschitz function
  $f_B$ from $B$ to $\widetilde{B}$.  We can also choose $f_B$ so that
  $f_B:\partial B \rightarrow \widetilde{B}$ is a homeomorphism from
  $\partial B$ to $\partial \widetilde{B}$.

  Both $L_1(n)$ and $L_2(n)$ depend only on $n$.
\end{lem}

We can now consider the map $\widetilde{H}: \widetilde{B} \rightarrow \widetilde{Y}$ given by $f_Y \circ G \circ f^{-1}_B$.
Since the maps are $2$-bilipschitz, $\widetilde{H}$ is still $\lesssim V$ bilipschitz.  With a slight abuse of notation, we will
refer to $\widetilde{Y}$ by $Y$, $\widetilde{H}$ by $H$, and $\widetilde{B}$ by $B$.
By using Proposition \ref*{prop:quantitative_simplicial_approximation}, we can subdivide the simplices
of $B$ to form $B'$ such that $H$ can be homotoped to a simplicial map from $B'$ to $Y$ with Lipschitz constant $\lesssim V$.
We also know that the side lengths of the simplices in $B'$ are $\gtrsim 1/V$.
We will define $Z$ to be the simplicial submanifold formed by applying $f_Y$ on the zero-bundle of $Th(n+k,n)$.

Clearly, $H^{-1}(Z) \cap \partial B$ is a PL manifold which is homeomorphic to $M$; this is because the map $f_B$ was assumed to be
a homeomorphism from the boundary of the ball to the boundary of the simplicial approximation of the ball.  We will begin
by perturbing $Z$ to $Z'$, a PL-manifold embedded in $Y$.  We want $Z'$ to have the following properties:
\begin{enumerate}
\item $Z'$ is an $n$-dimensional PL manifold.
\item $G^{-1}(Z') \cap \partial B$ is homeomorphic to $M$.
\item For every open $k$-simplex $c$ of $Y$, $Z'$ is transverse to $c$.
\item $Z'$ depends only on $n$.
\end{enumerate}
We can find such a PL-manifold by perturbing $Z$ using PL transversality theory.  There are several standard
references for this; see for example \cite[Theorem~5.3]{pl_topology}.  This theorem does not yield this result directly,
but can be adapted to do so.

We will use the transverse inverse image of $Z'$ to construct our filling.  We
know that $H^{-1}(Z') \cap \partial B$ is homeomorphic to $M$ from property (2).
Furthermore, the fact that the map is simplicial combined with properties (1)
and (3) implies that $H^{-1}(Z')$ is an $(n+1)$-dimensional PL manifold with
boundary, and its boundary is $H^{-1}(Z') \cap \partial B$.  Furthermore, since
the sphere, the ball, the simplicial approximations to them, and the embedded
manifold $\widetilde{M}$ are all orientable, from the discussion on page 210 of
\cite{milnor1974characteristic} we see that we also have that this manifold is
orientable, and agrees with the orientation of its boundary (which is
homeomorphic to $M$).

We now estimate the volume of $H^{-1}(Z^\prime)$.  Since $B$ only depends on $n$,
the number of simplices of $B^\prime$ is $\lesssim V^{n+k+1}$.  Since $H$ is a
simplicial map, the intersection of $H^{-1}(Z^\prime)$ with a given simplex
belongs to a finite set of subsets which depends only on $n$; since the
simplices are at scale $\sim 1/V$, the $(n+1)$-dimensional volume of this
intersection is $\lesssim V^{-(n+1)}$.  Therefore, the volume of $H^{-1}(Z^\prime)$
is $\lesssim V^k$, where $k$ is $O(\exp(n))$.

To build our manifold, we smooth out $W = H^{-1}(Z') \cap c$ and $\partial W$.
We can do this so that the volumes do not increase very much, and so that
$\partial{W}$, after smoothing, is diffeomorphic to $M$.  As above, since $Z'$
and $Y$ depend only on $n$, since $Y$ is a finite complex, and since the side
lengths of the simplices in $B$ are $\gtrsim 1/V$, this smoothing can be done so
that the result has $\geo \lesssim V$ (including on the boundary).  After
dilating the smoothed version of $W$ by a factor which is $\lesssim V$, we have
a compact oriented manifold $\widetilde{W}$ with $\geo(\widetilde W) \leq 1$,
whose boundary is (orientation-preserving) diffeomorphic to $M$.  The dilation
increases the volume of the resulting manifold by a factor of $\lesssim V^{n+1}$,
and so the result still has volume bounded by $\lesssim V^k$.

In particular, after the dilation has been performed, we obtain a manifold with
bounded local geometry with volume bounded by $\lesssim V^k$, and which
bounds a manifold diffeomorphic to $M$ with locally bounded geometry.  Thus the
complexity of the null-cobordism of $M$ is $\lesssim V^k$.  Since $V$ is within
a factor of $2$ of the complexity of $M$, this completes the proof of the
theorem.

\bibliographystyle{amsalpha}
\bibliography{qhomcob}

\end{document}